\documentclass{amsart}

\usepackage{amsmath, amsthm, amssymb, graphicx}
\usepackage{dsfont}
\usepackage{cite}
\usepackage{url}
\usepackage{bm}
\usepackage{multirow}
\usepackage{booktabs}
\usepackage{lineno}

\newtheorem{thm}{Theorem}
\newtheorem{lem}[thm]{Lemma}
\newtheorem{cor}[thm]{Corollary}

\newtheorem{assump}[thm]{Assumption}
\newtheorem{prop}[thm]{Proposition}

\newcommand{\HDivD}{H(\mathrm{div};\Omega)}
\newcommand{\HCurlD}{H(\mathbf{curl};\Omega)}
\newcommand{\HDiv}{H(\mathrm{div})}
\newcommand{\HCurl}{H(\mathbf{curl})}
\newcommand{\DIV}{\mathrm{div}\,}
\newcommand{\CURL}{\mathbf{curl}\,}
\newcommand{\ND}{\mathcal{ND}_{h}}
\newcommand{\RT}{\mathcal{RT}_{h}}
\newcommand{\NDH}{\mathcal{ND}_{H}}
\newcommand{\RTH}{\mathcal{RT}_{H}}

\newcommand{\Sh}{\mathcal{S}_{h}}

\newcommand{\bg}{\bm{g}}
\newcommand{\bn}{\bm{n}}
\newcommand{\bp}{\bm{p}}
\newcommand{\bq}{\bm{q}}
\newcommand{\br}{\bm{r}}
\newcommand{\bt}{\bm{t}}
\newcommand{\bu}{\bm{u}}
\newcommand{\bv}{\bm{v}}
\newcommand{\bw}{\bm{w}}
\newcommand{\bx}{\bm{x}}
\newcommand{\cT}{\mathcal{T}}
\newcommand{\cF}{\mathcal{F}}
\newcommand{\cE}{\mathcal{E}}
\newcommand{\znorm}[2]{\left\| #1 \right\|_{#2}}
\newcommand{\mnorm}[3]{\left\| #1 \right\|_{#2,#3}}
\newcommand{\seminorm}[3]{\left \vert #1 \right \vert_{#2, #3}}

\def\ad#1{\begin{aligned}#1\end{aligned}}  \def\b#1{\bm{#1}} 
\def\a#1{\begin{align*}#1\end{align*}} \def\an#1{\begin{align}#1\end{align}} \def\t#1{\hbox{#1}}
 
\def\p#1{\begin{pmatrix}#1\end{pmatrix}} \def\bb#1{{\mathbf{#1}}}
  \numberwithin{equation}{section}
\numberwithin{table}{section} \numberwithin{figure}{section}

\title[OS for Vector Fields]{New Analysis of Overlapping Schwarz Methods for Vector Field Problems in Three Dimensions with Generally Shaped Domains. }

\author{Duk-Soon Oh}
\address{Department of Mathematics, Chungnam National University, Daejeon, 34134, Republic of Korea.}
\email{duksoon@cnu.ac.kr}
\author{Shangyou Zhang}
\address{Department of Mathematical Sciences, University of Delaware, Newark, DE 19716, USA}
\email{szhang@udel.edu}
\date{\today}

\begin{document}

\begin{abstract}
This paper introduces a novel approach to analyzing overlapping Schwarz methods for N\'{e}d\'{e}lec and Raviart--Thomas vector field problems. The theory is based on new regular stable decompositions for vector fields that are robust to the topology of the domain. Enhanced estimates for the condition numbers of the preconditioned linear systems are derived, dependent linearly on the relative overlap between the overlapping subdomains. Furthermore, we present the numerical experiments which support our theoretical results.
\end{abstract}

    

\maketitle

\section{Introduction}\label{sec:introduction}
Let $\Omega$ be a bounded Lipschitz domain in $\mathds{R}^3$. We assume that the domain $\Omega$ is scaled such that the diameter of $\Omega$ is equal to one. We first introduce the Hilbert space $\HCurlD$ that consists of square integrable vector fields on the doamin $\Omega$ that have square integrable curls. 
We consider the following model problem posed in $\HCurlD$: Find $\bu \in \HCurlD$ such that 
\begin{equation}\label{eq:H(curl) model problem}
    a_c (\bu, \bv) = (\bm{f}, \bv) \quad \forall \bv \in \HCurlD, 
\end{equation}
where 
\begin{equation}\label{eq:H(curl) bilinear form}
    a_c (\bu, \bv) := \eta_c \left( \CURL \bu, \CURL \bv \right) + \left(\bu, \bv \right)
\end{equation}
and $\left( \cdot, \cdot \right)$ is the standard inner product on $\left(L^2(\Omega)\right)^3$ or $L^2(\Omega)$. We assume that the constant $\eta_c$ is positive and $\bm{f} \in \left(L^2(\Omega)\right)^3$. We also consider the Hilbert space $\HDivD$ in a similar manner, i.e., the space of square integrable vector fields on $\Omega$ with square integrable divergences. The corresponding model problem for a square integrable vector field $\bg$ on $\Omega$ is given as follows: Find $\bp \in \HDivD$ such that
\begin{equation}\label{eq:H(div) model problem}
    a_d (\bp, \bq) = (\bg, \bq) \quad \forall \bq \in \HDivD,
\end{equation}
where
\begin{equation}\label{eq:H(div) bilinear form}
    a_d (\bp, \bq) := \eta_d \left( \DIV \bp, \DIV \bq \right) + \left(\bp, \bq \right).
\end{equation}
Similarly, we assume that $\eta_d$ is a positive constant.

The first model problem \eqref{eq:H(curl) model problem} is originated from time-dependent Maxwell's equation, specifically the eddy-current problem; see \cite{MR2143847, MR1135758}. With a suitable time discretization, we have to solve the problem \eqref{eq:H(curl) model problem} in each time step. The second problem \eqref{eq:H(div) model problem} is developed for a first-order system of least-squares formulation for standard second order elliptic problems. For more detail, see \cite{MR1302685}. We also note that efficient numerical solution methods related to \eqref{eq:H(div) model problem} are required for solving problems from a pseudostress-velocity formulation for the Stokes equations and a sequential regularization method for the Navier-Stokes equations; see \cite{MR2642330, MR1451113}.

There have been a number of attempts to develop domain decomposition methods for solving \eqref{eq:H(curl) model problem} and \eqref{eq:H(div) model problem}; see \cite{MR4064357, MR3454359, MR3465088, MR3242973, MR2193972, MR2169505, MR1838270, MR1794350, MR2103209, MR3989859, MR3718365, LXZ:2023:OSSharp, MR2035002} for $\HCurl$ problems and \cite{MR3739213, MR1838270, MR3033012, MR3243012, MR3618847, MR1759911, LXZ:2023:OSSharp, MR1401938} for $\HDiv$ problems. Other types of fast solution methods, such as multigrid methods, have been also suggested. For more details, see \cite{MR1401938, MR1754719, MR1615161, MR1654571, MR2361899, MR2536904, MR3029843, MR2372343, MR3859167, MR3989901, Oh:2022:MGHCurl, Oh:MGHcurlNE,
 BF:2022:MGdeRham}.

The framework for analyzing domain decomposition methods based on overlapping subdomains has been introduced in \cite{MR1193013} as a subspace correction method. The two-level overlapping Schwarz methods for scalar elliptic problems have been introduced and analyzed in \cite{MR1273155}; see also \cite[Section 3]{MR2104179} and references therein for more detailed techniques. In \cite{MR1273155}, it is proved that the condition number of the preconditioned linear system is bounded above by a constant multiple of $(1 + H/\delta)$, where $H$ is the diameter of the subdmain and $\delta$ is the size of the overlap between subdomains. In fact, the bound is shown to be optimal; see \cite{MR1771050}.

The purpose of this paper is to analyze two-level overlapping Schwarz methods for discretized problems originated from \eqref{eq:H(curl) model problem} and \eqref{eq:H(div) model problem} using appropriate finite elements, i.e., N\'{e}d\'{e}lec and Raviart--Thomas elements of the lowest order. Such methods have been first introduced and analyzed in \cite{MR1838270, MR1794350}. The authors in \cite{MR1838270, MR1794350} provided an upper bound with the quadratic dependence on $(1 + H/\delta)$ with the assumption, i.e., the convexity of the domain. Later, the first author of this paper suggested an upper bound, $C (1 + \log(H/h))(1 + H / \delta)$, in \cite{MR3033012} with a nonstandard coarse space method assuming that subdomains are convex, where $h$ is the size of the mesh for the finite elements. Recently, an improved bound that depends linearly on $(1 + H/\delta)$ has been obtained in \cite{LXZ:2023:OSSharp} under the same assumption with \cite{MR1838270, MR1794350}. In this paper, we will provide a linear bound without any assumptions related to the topological properties of the domain and subdomains. We remark that the algorithms in \cite{MR1838270, MR1794350, LXZ:2023:OSSharp} and this paper are essentially the same but the technical details for the theories are different.

The important ingredients for analyzing numerical methods for solving problems posed in $\HCurl$ and $\HDiv$ are the Helmholtz type decompositions. This is because the structures of the kernels of the curl and the divergence operators are quite different from that of the gradient operator. In \cite{MR1838270, MR1794350}, discrete orthogonal Helmholtz decompositions based on those for continuous spaces have been suggested and used for analyzing overlapping Schwarz methods. Since the discrete range spaces are not included the continuous range spaces, the authors had to introduce semi-continuous spaces to handle the difficulty. To do so, the convexity of the domain was needed to use a suitable embedding. In \cite{LXZ:2023:OSSharp}, the authors considered the same type of decompositions so that the assumption for the domain has been inherited. In this paper, we consider a different type of regular decompositions. By introducing an additional term, an oscillation component, and abandoning the orthogonality, we have more robust decompositions. The approaches have been originally introduced in \cite{MR2361899} and extended later in \cite{MR4398318, HP:2019:RegularDecomp, MR4143285} based on the cochain projections constructed in \cite{MR3246803}. Our theories will be based on the decompositions suggested by Hiptmair and Pechstein; see \cite{MR4398318, HP:2019:RegularDecomp, MR4143285}.

The rest of the paper is organized as follows. In Section~\ref{sec:model problems and finite elements}, we introduce the discrete model problems and related finite elements. We describe overlapping Schwarz preconditioners in Section~\ref{sec:overlapping schwarz methods}. We next provide our theoretical results in Section~\ref{sec:condition number estimate}. Finally, the numerical examples to support our theories are presented in Section~\ref{sec:numerical experiments}.

\section{The discrete problems}\label{sec:model problems and finite elements}
We consider two triangulations, $\cT_H$ and $\cT_h$. First, we introduce $\cT_H$, a coarse triangulation of the domain $\Omega$, consisting of shape-regular and quasi-uniform tetrahedral elements with a maximum diameter $H$. Subsequently, $\cT_h$ is generated as a finer mesh, a refinement of the coarse mesh $\cT_H$. It is assumed that the restriction of $\cT_h$ to each individual coarse element is both shape-regular and quasi-uniform. 

We next introduce finite element spaces. The space of the lowest order tetrahedral N\'{e}d\'{e}lec finite elements associated with $\HCurlD$ and the triangulation $\cT_h$ is defined by
\[
    \ND :=  \left\{ \bu \,\, \vert \,\, \bu_{|K} \in N(K) , K \in \cT_h \,\, \mathrm{and} \,\, \bu \in \HCurlD \right\},
\]
where the shape function $N(K)$ is given by
\[
    N(K) := \bm{\alpha}_c + \bm{\beta}_c \times \bx
\]
for a tetrahedral element. Here, $\bm{\alpha}_c$ and $\bm{\beta}_c$ are constant vectors in $\mathds{R}^3$. The values of the two vectors $\bm{\alpha}_c$ and $\bm{\beta_c}$ can be determined by the average tangential components on the edges of $K$, i.e., 
\[ 
    \lambda_e^{\mathcal{ND}} (\bu) := \dfrac{1}{|e|} \int_e \bu \cdot \bt_e \, ds, \quad e \subset \partial K,
\] 
where $|e|$ is the length of the edge $e$ and $\bt_e$ is the unit tangential vector associated with $e$. We note that these values can be considered as the degrees of freedom. The interpolation operator $\Pi_h^{\mathcal{ND}}$ for a sufficiently smooth vector field $\bu$ in $\HCurlD$ onto $\ND$ is defined as follows:
\[ 
    \Pi_h^{\mathcal{ND}} \bu := \sum_{e \in \cE_h} \lambda_e^{\mathcal{ND}} (\bu) \, \Phi_e^{\mathcal{ND}},
\] 
where $\cE_h$ is the set of interior edges of $\cT_h$ and $\Phi_e^{\mathcal{ND}}$ is the standard basis function linked with $e$.

We next consider the lowest order tetrahedral Raviart--Thomas finite element space corresponding to the space $\HDivD$ that is defined by
\[ 
    \RT :=  \left\{ \bp \,\, \vert \,\, \bp_{|K} \in R(K) , K \in \cT_h \,\, \mathrm{and} \,\, \bp \in \HDivD \right\}.
\] 
Here, the shape function $R(K)$ associated with the tetrahedral element $K$ is defined by 
\[ 
    R(K) := \bm{\alpha}_d + \beta_d \, \bx,
\] 
where $\bm{\alpha}_d$ is a constant vector in $\mathds{R}^3$ and $\beta_d$ is a scalar. The degrees of freedom related to an element $K$ are determined by the average values of the normal components over its faces, namely 
\[ 
    \lambda_f^{\mathcal{RT}} (\bp) := \dfrac{1}{|f|} \int_f \bp \cdot \bn_f \, ds, \quad f \subset \partial K.
\] 
Here, $|f|$ is the area of the face $f$ and $\bn_f$ is the unit normal vector corresponding to $f$. We note that $\bm{\alpha}_d$ and $\beta_d$ can be completely recovered by the degrees of freedom associated with the four faces of $K$. Let $\cF_h$ be the set of interior faces of $\cT_h$. Similarly, we can define the interpolation operator $\Pi_h^{\mathcal{RT}}$ associated with $\HDivD$. For a sufficiently smooth $\bu \in \HDivD$, the operator is defined by
\[ 
    \Pi_h^{\mathcal{RT}} \bu := \sum_{f \in \cF_h} \lambda_f^{\mathcal{RT}} (\bu) \, \Phi_f^{\mathcal{RT}}.
\] 
Here, $\Phi_f^{\mathcal{RT}}$ is the standard basis function corresponding to the face $f$. 

In addition, we need the piecewise linear space for our theories. Let $\Sh$ be the space of the continuous $P_1$ finite elements associated with $\cT_h$. We recall that the degrees of freedom are given by the function evaluations at the vertices. The corresponding interpolation operator for a sufficiently smooth function in $H^1(\Omega)$ is given by $\Pi_h^{\mathcal{S}}$. We also consider $\widetilde{\Pi}_h^{\mathcal{S}}$, the Scott-Zhang interpolation operator introduced in \cite{MR1011446}. We can also consider the interpolation operators for $\cT_H$ by replacing the subscript with $H$.

By restricting the model problems \eqref{eq:H(curl) model problem} and \eqref{eq:H(div) model problem} to the finite element spaces $\ND$ and $\RT$, respectively, we obtain the following discrete problems: Find $\bu_h \in \ND$ such that
\[ 
    a_c (\bu_h, \bv_h) = (\bm{f}, \bv_h) \quad \forall \bv_h \in \ND 
\] 
and
find $\bp_h \in \RT$ such that
\[ 
    a_d (\bp_h, \bq_h) = (\bg, \bq_h) \quad \forall \bq_h \in \RT. 
\] 
We also define the operators $A_c : \ND \rightarrow \ND$ and $A_d : \RT \rightarrow \RT$ as follows:
\[ 
    \left(A_c \bu_h, \bv_h \right) = a_c (\bu_h, \bv_h) \quad \forall \bu_h, \bv_h \in \ND
\] 
and
\[ 
    \left(A_d \bp_h, \bq_h \right) = a_d (\bp_h, \bq_h) \quad \forall \bp_h, \bq_h \in \RT.
\] 
\section{Overlapping Schwarz methods}\label{sec:overlapping schwarz methods}
We decompose the domain $\Omega$ into $N$ nonoverlapping subdomains $\Omega_i$, a union of a few elements in $\cT_H$. We assume that the number of coarse elements contained in each subdomain is uniformly bounded. The parameter $H_i$ is defined by the diameter of the subdomain $\Omega_i$. We now consider an overlapping subdomain $\Omega_i'$ originated from the nonoverlapping subdomain $\Omega_i$ by extending layers of fine elements, i.e., $\Omega_i'$ containing $\Omega_i$ is a union of fine elements. In addition, we consider the assumptions introduced in \cite[Assumptions 3.1, 3.2, and 3.5]{MR2104179}.
\begin{assump}\label{assump:delta}
    For $i = 1, 2, \cdots, N$, there exists $\delta_i > 0$, such that, if $\bx$ belongs to $\Omega_i'$, then 
    \begin{equation}
        \mathrm{dist}(\bx, \partial \Omega_j' \setminus \partial \Omega) \ge \delta_i,
    \end{equation}
    for a suitable $j = j(\bx)$, possibly equal to $i$, with $\bx \in \Omega_j'$.
\end{assump}

\begin{assump}\label{assump:finite covering}
    The partition $\left\{ \Omega_i' \right\}$ can be colored using at most $N_0$ colors, in such a way that subregions with the same color are disjoint.
\end{assump}

\begin{assump}\label{assump:H}
There exists a constant $C$ independent of $\cT_H$ and the subdomain $\Omega_i'$, such that, for $i = 1, 2, \cdots, N$,
\begin{equation}
    H_K \le C H_i,
\end{equation}
for any $K \in \cT_H$, such that $K \cap \Omega_i' \neq \emptyset$. Here, $H_K$ is the diameter of the coarse element $K$.
\end{assump}

In our theories, a partition of unity technique plays an essential role. To do so, we construct the set $\left\{ \theta_i \right\}$, consisting of piecewise linear functions associated with the overlapping subdomain, which has the following properties:
\begin{equation}\label{eq:property theta_i}
    \begin{aligned}
        & 0 \le \theta_i \le 1, \\
        & \mbox{supp } (\theta_i) \subset \overline{\Omega_i'}, \\
        & \sum_{i=1}^N \theta_i \equiv 1, \quad \bx \in \Omega, \\
        & \znorm{\theta_i}{\infty} \le \dfrac{C}{\delta_i},
    \end{aligned}
\end{equation} 
where $C$ is a constant independent of the $\delta_i$ and the $H_i$ and $\znorm{\cdot}{\infty}$ is the standard $L^{\infty}-$norm.
For more details, see \cite[Lemma 3.4]{MR2104179}.

We now construct our preconditioners based on overlapping Schwarz methods. We first consider the coarse component. The coarse operators $A_{c}^{(0)}$ and $A_{d}^{(0)}$ related to the coarse problems are defined as follows:
\[ 
    \left(A_{c}^{(0)} \bu_H, \bv_H \right) = a_c (\bu_H, \bv_H) \quad \forall \bu_H, \bv_H \in \NDH
\] 
and
\[ 
    \left(A_{d}^{(0)} \bp_H, \bq_H \right) = a_d (\bp_H, \bq_H) \quad \forall \bp_H, \bq_H \in \RTH.
\] 
The components of the operator $R_{c}^{(0)}$ which maps a vector field in $\ND$ to $\NDH$ consist of the coefficients obtained through the interpolation of the standard basis functions associated with $\NDH$ onto the mesh $\cT_h$. 
We remark that ${R_{c}^{(0)}}^T : \NDH \rightarrow \ND$ is the natural injection since the finite element spaces are nested. 
In a similar way, we can define the operator $R_{d}^{(0)} : \RT \rightarrow \RTH$ associated with the Raviart--Thomas spaces.  

Regarding the local components, let us define the restriction operators 
  $R_{c}^{(i)} : \ND \rightarrow \ND^{(i)}$  in such a way that ${R_{c}^{(i)}}^T : \ND^{(i)} \rightarrow \ND$ are natural injections. Here, $\ND^{(i)}$ is the subspace of $\ND$ spanned by the basis functions corresponding to the fine edges in $\Omega_i'$.
Similarly, the construction for $R_{d}^{(i)} : \RT \rightarrow \RT^{(i)}$ is straightforward, where the local space $\RT^{(i)}$ is defined in a similar way. 
Then, the local operators $A_{c}^{(i)}$ and $A_{d}^{(i)}$ can be defined as follows: 
\[
    A_{\xi}^{(i)} = R_{\xi}^{(i)} A_{\xi} {R_{\xi}^{(i)}}^T,
\]
where $\xi$ corresponds $c$ or $d$. We note that $A_{\xi}^{(i)}$ is just a principal minor of $A_{\xi}$.

We can now construct the preconditioners and the resulting preconditioned linear operator has the following form:
\begin{equation}\label{eq:preconditioned linear operator}
    M_{\xi}^{-1}A_{\xi} = \sum_{i=0}^N {R_{\xi}^{(i)}}^T {A_{\xi}^{(i)}}^{-1} R_{\xi}^{(i)} A_{\xi},
\end{equation}
where $\xi$ corresponds $c$ or $d$.
\section{Condition number estimate}\label{sec:condition number estimate}
\subsection{Preliminaries}\label{subsec:preliminaries}
In this subsection, we will describe several preliminary results for our theories. 

We first consider standard Sobolev spaces and their norms and semi-norms. For any $\mathcal{D} \subset \Omega$, let us denote by $\mnorm{\, \cdot\, }{s}{\mathcal{D}}$ and $\seminorm{\, \cdot\, }{s}{\mathcal{D}}$ the norm and the semi-norm of the Sobolev space $H^s(\mathcal{D})$, respectively. Provided that $\mathcal{D} = \Omega$, we will omit the subscript $\Omega$ for convenience. If there is no explicit confusion, the same norm and semi-norm notations will be used for $\left(H^s(\mathcal{D})\right)^3$.

We next define the operator $Q_H^{\mathcal{ND}} : \left(L^2(\Omega)\right)^3 \rightarrow \NDH$ as the $L^2-$projection onto $\NDH$. Similarly, we define the $L^2-$projection operator $Q_H^{\mathcal{RT}} : \left(L^2(\Omega)\right)^3 \longrightarrow \RTH$. We then have the following lemma in \cite[Chapter 10]{MR2104179}:
\begin{lem}\label{lem:L2 projection}
    For $\bu, \bp \in \left(H^1(\Omega)\right)^3$, the following estimates hold:
    \[ 
        \begin{aligned}
            \znorm{\CURL \left(Q_H^{\mathcal{ND}} \bu \right)}{0} & \le C \left \vert \bu \right \vert_1, \\
            \znorm{\bu -  Q_H^{\mathcal{ND}} \bu}{0} & \le C H \left \vert \bu \right \vert_1, \\
            \znorm{\DIV \left(Q_H^{\mathcal{RT}} \bp \right)}{0} & \le C \left \vert \bp \right \vert_1, \\
            \znorm{\bp -  Q_H^{\mathcal{RT}} \bp}{0} & \le C H \left \vert \bp \right \vert_1, \\
        \end{aligned}
    \] 
    with constants independent of $\bu$, $\bp$, and $H$.
\end{lem}

We also denote by $Q_{H, K}^0 : \left(L^2(K)\right)^3 \rightarrow \left(P_0(K)\right)^3$, where $K \in \cT_H$ and $P_0(K)$ is the space of constants, a local $L^2-$projection operator. Then, we have
\begin{lem}\label{lem:L2 projection K}
    Let $K \in \cT_H$. Then, for $\bu \in \left(H^1(K)\right)^3$, we have
    \[ 
        \mnorm{\bu - Q_{H, K}^0 \bu}{0}{K} \le C H_K \seminorm{\bu}{1}{K},
    \] 
    where $H_K$ is the diameter of $K$.
\end{lem}

The following lemma describes the stability of the interpolation operators, stated in \cite[Chapter 10]{MR2104179}, for the functions obtained by the product of a piecewise linear function and a vector field:
\begin{lem}\label{lem:interpolation with theta_i}
    Let $\bu \in \ND$, $\bp \in \RT$, and $\theta_i$ be any continuous, piecewise linear function supported in the subdomain $\Omega_i'$. Then, we have the following estimates:
    \[ 
        \begin{aligned}
            \mnorm{\Pi_h^{\mathcal{ND}}\left(\theta_i \bu \right)}{0}{\Omega_i'} & \le C \mnorm{\theta_i \bu}{0}{\Omega_i'}, \\
            \mnorm{\CURL\left(\Pi_h^{\mathcal{ND}}\left(\theta_i \bu \right)\right)}{0}{\Omega_i'} & \le C \mnorm{\CURL\left(\theta_i \bu\right)}{0}{\Omega_i'}, \\
            \mnorm{\Pi_h^{\mathcal{RT}}\left(\theta_i \bp \right)}{0}{\Omega_i'} & \le C \mnorm{\theta_i \bp}{0}{\Omega_i'}, \\
            \mnorm{\DIV\left(\Pi_h^{\mathcal{RT}}\left(\theta_i \bp \right)\right)}{0}{\Omega_i'} & \le C \mnorm{\DIV\left(\theta_i \bp\right)}{0}{\Omega_i'}. \\
        \end{aligned}
    \] 
\end{lem}

We finally introduce an estimate for piecewise $H^1$ functions on the layer around the subdomain $\Omega_i$. A similar estimate for $H^1$ functions is given in \cite[Lemma 3.10]{MR2104179} and extended for piecewise $H^1$ functions in \cite[Lemma 3.3]{LXZ:2023:OSSharp}. In \cite{LXZ:2023:OSSharp}, the author assumed that the subdomain $\Omega_i \in \cT_H$. We note that the estimate is equally valid with assumptions in the beginning of Section~\ref{sec:overlapping schwarz methods} together with Assumptions~\ref{assump:delta}, \ref{assump:finite covering} and \ref{assump:H}. Before we consider the estimate, we define $\Omega_{i, \delta}$ by
\[  
    \Omega_{i, \delta} = \displaystyle{\bigcup_{\substack{j \in I_i \\ j \neq i}} \Omega_i' \cap \Omega_j'},
\] 
where $I_i = \left\{ j : \Omega_i' \cap \Omega_j' \neq \emptyset \right\}$.

\begin{lem}\label{lem:est layer Omega_i}
    Let $\bu$ be a piecewise $H^1$ function, i.e., $\bu\vert_{K} \in \left(H^1(K)\right)^3$ on each $K \in \cT_H$. We then have
    \[ 
        \delta_i^{-2} \mnorm{\bu}{0}{\Omega_{i, \delta}}^2 \le C \sum_{j \in I_i} \sum_{\substack{K \in \cT_H, \\ K \subset \Omega_j}}\left[\left(1 + \dfrac{H_i}{\delta_i}\right)\seminorm{\bu}{1}{K}^2 + \dfrac{1}{\delta_i H_i} \mnorm{\bu}{0}{K}^2\right].
    \] 
\end{lem}

\subsection{Regular decompositions for vector fields}\label{subsec:regular decomposition}
We first introduce the cochain projections introduced in \cite{MR3246803} and extended in \cite{HP:2019:RegularDecomp, MR4398318}. Let 
\[ 
    \begin{aligned}
        \pi_h^{\mathcal{ND}} & : \HCurlD \rightarrow \ND, \\
        \pi_h^{\mathcal{RT}} & : \HCurlD \rightarrow \RT, \\
        \mbox{and } \pi_h^{0} & : L^2(\Omega) \rightarrow P_0(\Omega)
    \end{aligned}
\] 
denote the cochain projection operators constructed in \cite{MR3246803} and \cite{HP:2019:RegularDecomp}. We note that the operators satisfy the commuting properties on each element in $\cT_h$
\begin{equation}\label{eq:commuting properties}
    \begin{alignedat}{2}
        \CURL \left(\pi_h^{\mathcal{ND}} \bu \right) & = \pi_h^{\mathcal{RT}} \left( \CURL \bu \right) \quad && \forall \bu \in \HCurlD, \\
        \DIV \left(\pi_h^{\mathcal{RT}} \bp \right) & = \pi_h^0 \left( \DIV \bp \right) \quad && \forall \bp \in \HDivD
    \end{alignedat}
\end{equation}
and the local stability estimates
\begin{equation}\label{eq:cochain local stability}
    \begin{alignedat}{2}
        \mnorm{\pi_h^{\mathcal{ND}} \bu}{0}{K} & \le C \left( \mnorm{\bu}{0}{\omega_K} + h_K \mnorm{\CURL \bu}{0}{\omega_K}\right) \quad && \forall \bu \in \HCurlD, \\
        \mnorm{\pi_h^{\mathcal{RT}} \bp}{0}{K} & \le C \left( \mnorm{\bp}{0}{\omega_K} + h_K \mnorm{\DIV \bp}{0}{\omega_K}\right) \quad && \forall \bp \in \HDivD, \\
        \mnorm{\pi_h^0 z}{0}{K} & \le C \mnorm{z}{0}{\omega_K} \quad  && \forall z \in L^2(\Omega),
    \end{alignedat}
\end{equation}
where $K \in \cT_h$, $h_K$ is the diameter of $K$, and $\omega_K$ is the union of the neighboring elements of $K$.
We also remark that the fact that $\bu_h = \pi_h^{\mathcal{ND}} \bu_h, \forall \bu_h \in \ND$ and $\bp_h = \pi_h^{\mathcal{RT}} \bp_h, \forall \bp_h \in \RT$, the inverse inequality, and \eqref{eq:cochain local stability} ensure the estimates
\[ 
        \znorm{\bw_h - \pi_h^{\mathcal{ND}}\bw_h}{0} \le C h \znorm{\bm{\nabla} \bw_h}{0}
\] 
and
\[ 
    \znorm{\bw_h - \pi_h^{\mathcal{RT}}\bw_h}{0} \le C h \znorm{\bm{\nabla} \bw_h}{0}
\] 
for all $\bw_h \in (\Sh)^3$.

We next consider the following regular decomposition in \cite[Theorem 10]{MR4143285} for edge elements:
\begin{lem}[Hiptmair-Pechstein decomposition for edge elements]\label{lem:HP decomp ND}
    For each $\bu_h \in \ND$,  there exist a continuous and piecewise linear scalar function $\chi_h \in \Sh$, a continuous and piecewise linear vector field $\bw_h \in (\Sh)^3$, and a remainder $\widetilde{\bu}_h \in \ND$, all depending linearly on $\bu_h$, providing the discrete regular decomposition
    \begin{equation}\label{eq:ND regular decomp}
        \bu_h = \nabla \chi_h + \pi_{h}^{\mathcal{ND}} \bw_h + \widetilde{\bu}_h
    \end{equation}
    and satisfying the stability estimates
    \begin{align}
        \znorm{\nabla \chi_h}{0} + \znorm{\bw_h}{0} + \znorm{\widetilde{\bu}_h}{0} & \le C \znorm{\bu_h}{0}\label{eq:ND decomp est 1} \\
        \znorm{\bm{\nabla} \bw_h}{0} + \znorm{h^{-1} \widetilde{\bu}_h}{0} & \le C \left(\znorm{\CURL \bu_h}{0} +  \znorm{\bu_h}{0} \right),\label{eq:ND decomp est 2}
    \end{align}
    where $C$ is a generic constant that depends only on the shape of $\Omega$, but not on the shape-regularity constant of $\cT_h$. Here, $h^{-1}$ is the piecewise constant function that is equal to $h_K^{-1}$ on every element $K \in \cT_h$. 
\end{lem}

The regular decomposition in \cite[Theorem 13]{MR4143285} for face elements is given in the following lemma:
\begin{lem}[Hiptmair-Pechstein decomposition for face elements]\label{lem:HP decomp RT}
    For each $\bp_h \in \RT$, there exist a vector field $\bm{\rho}_h \in \ND$, a continuous and piecewise linear vector field $\br_h \in (\Sh)^3$, and a remainder $\widetilde{\bp}_h \in \RT$, all depending linearly on $\bp_h$, providing the discrete regular decomposition
    \begin{equation}\label{eq:RT regular decomp}
        \bp_h = \CURL \bm{\rho}_h + \pi_{h}^{\mathcal{RT}} \br_h + \widetilde{\bp}_h
    \end{equation}
    with the bounds
    \begin{align}
        \znorm{\CURL \bm{\rho}_h}{0} + \znorm{\bm{\rho}_h}{0} + \znorm{\br_h}{0} + \znorm{\widetilde{\bp}_h}{0} & \le C \znorm{\bp_h}{0}\label{eq:RT decomp est 1} \\
        \znorm{\bm{\nabla} \br_h}{0} + \znorm{h^{-1} \widetilde{\bp}_h}{0} & \le C \left(\znorm{\DIV \bp_h}{0} + \znorm{\bp_h}{0}\right),\label{eq:RT decomp est 2}
    \end{align}
    where $C$ is a generic constant that depends only on the shape of $\Omega$, but not on the shape-regularity constant of $\cT_h$. Here, $h^{-1}$ is the piecewise constant function that is equal to $h_K^{-1}$ on every element $K \in \cT_h$.
\end{lem}

We note that the Hiptmair-Pechstein decompositions are a generalized version of those in \cite{MR2361899} and good alternatives of discrete Helmholtz decompositions introduced in \cite{MR1794350, LXZ:2023:OSSharp} since they are more robust to the topology of the domain.

\subsection{Schwarz framework} In this subsection, we summarize the abstract Schwarz framework, a key ingredient for analyzing domain decomposition methods. For more detail, see \cite[Chapter 2]{MR2104179}.
\begin{lem}\label{lem:smallest eigenvalue ND}
    If for all $\bu_h \in \ND$ there is a representation, $\bu_h = \sum_{i=0}^N \bu_i$, where $\bu_0 \in \NDH$ and $\bu_i \in \ND^{(i)}$ for $i = 1, 2, \cdots, N$, such that
    \[ 
        \sum_{i=0}^N a_c(\bu_i, \bu_i) \le C_c^2 a_c(\bu_h, \bu_h),
    \] 
    then the smallest eigenvalue of the preconditioned linear operator defined in \eqref{eq:preconditioned linear operator} is bounded from below by $C_c^{-2}$.
\end{lem}
\begin{lem}\label{lem:smallest eigenvalue RT}
    If for all $\bp_h \in \RT$ there is a representation, $\bp_h = \sum_{i=0}^N \bp_i$, where $\bp_0 \in \RTH$ and $\bp_i \in \RT^{(i)}$ for $i = 1, 2, \cdots, N$, such that
    \[ 
        \sum_{i=0}^N a_d(\bp_i, \bp_i) \le C_d^2 a_d(\bp_h, \bp_h),
    \] 
    then the smallest eigenvalue of the preconditioned linear operator defined in \eqref{eq:preconditioned linear operator} is bounded from below by $C_d^{-2}$.
\end{lem}

\begin{lem}\label{lem:largest eigenvalue}
The largest eigenvalue of the operator introduced in \eqref{eq:preconditioned linear operator} is bounded from above by $(N_0 + 1)$, where $N_0$ is defined in Assumption~\ref{assump:finite covering}.
\end{lem}

\subsection{Condition number estimate for $\HCurl$}
Based on Lemma~\ref{lem:HP decomp ND}, for any $\bu_h \in \ND$, we can find $\chi_h$, $\bw_h$, and $\widetilde{\bu}_h$, which satisfy \eqref{eq:ND decomp est 1} and \eqref{eq:ND decomp est 2}. We then consider
\begin{equation}\label{eq:ND stable decomposition 1}
    \begin{aligned}
        \bu_0 & := \nabla \chi_0 + \bw_0, \\
        \bu_i & := \nabla \chi_i + \bw_i + \widetilde{\bu}_i, \quad i = 1, 2, \cdots, N, 
    \end{aligned}
\end{equation}
where
\begin{equation}\label{eq:ND stable decomposition 2}
    \begin{aligned}
        \chi_0 & = \widetilde{\Pi}_H^{\mathcal{S}} \chi_h, \\
        \bw_0 & = Q_H^{\mathcal{ND}} \bw_h, \\
        \chi_i & = \Pi_h^{\mathcal{S}} (\theta_i(\chi_h - \chi_0)), \\
        \bw_i & = \Pi_{h}^{\mathcal{ND}}(\theta_i(\pi_{h}^{\mathcal{ND}} \bw_h - \bw_0)), \\
        \widetilde{\bu}_i & = \Pi_{h}^{\mathcal{ND}}(\theta_i \widetilde{\bu}_h).
    \end{aligned}
\end{equation}
Here, the interpolation operators $\widetilde{\Pi}_H^{\mathcal{S}}$, $\Pi_h^{\mathcal{S}}$, and $\Pi_{h}^{\mathcal{ND}}$ are defined in Section~\ref{sec:model problems and finite elements} and the set $\left\{ \theta_i \right\}$, the $L^2-$projection operator $Q_H^{\mathcal{ND}}$, and the cochain projection $\pi_h^{\mathcal{ND}}$ are mentioned in Section~\ref{sec:overlapping schwarz methods}, \ref{subsec:preliminaries}, and \ref{subsec:regular decomposition}, respectively. From \eqref{eq:ND stable decomposition 1} and \eqref{eq:ND stable decomposition 2}, we can easily check $\bu_0 \in \NDH, \bu_i \in \ND^{(i)},$ and $\bu_h = \sum_{i=0}^N \bu_i$. We separately estimate the coarse component $\bu_0$ and the local components, i.e.,  $\bu_i, i = 1, \cdots, N$.

We first consider the coarse component. The next lemma shows the stability of $\bu_0$.
\begin{lem}\label{lem:ND coarse estimate}
    Assume that the constant $\eta_c$ in \eqref{eq:H(curl) bilinear form} is less than or equal to one. Then, we have the following estimate for the coarse component in \eqref{eq:ND stable decomposition 1}:
    \begin{equation}\label{eq:ND coarse estimate}
        a_c(\bu_0, \bu_0) \le C a_c(\bu_h, \bu_h),
    \end{equation}
    where the constant $C$ does not depend on $N$, $h$, $H_i$, $\delta_i$, and $\eta_c$.
\end{lem}
\begin{proof}
We note that $\bu_0 = \nabla \chi_0 + \bw_0$. We estimate each term separately.
\begin{itemize}
    \item Term $\chi_0$:
    
    From the property of Scott-Zhang interpolation and \eqref{eq:ND decomp est 1}, we have
    \begin{equation}\label{eq:est chi_0}
        a_c(\nabla \chi_0, \nabla \chi_0) = \znorm{\nabla \chi_0}{0}^2 \le C \znorm{\nabla \chi_h}{0}^2 \le C\znorm{\bu_h}{0}^2.
    \end{equation}
    \item Term $\bw_0$:
    
    By using the definition of $Q_H^{\mathcal{ND}}$ and \eqref{eq:ND decomp est 1}, we obtain
    \begin{equation}\label{eq:est w_0}
        \znorm{\bw_0}{0}^2 = \znorm{Q_H^{\mathcal{ND}} \bw_h}{0}^2 \le \znorm{\bw_h}{0}^2 \le C \znorm{\bu_h}{0}^2.
    \end{equation}
    Due to Lemma~\ref{lem:L2 projection} and \eqref{eq:ND decomp est 2}, we have
    \begin{equation}\label{eq:est curl w_0}
        \begin{aligned}
        \eta_c \znorm{\CURL \bw_0}{0}^2 & = \eta_c \znorm{\CURL (Q_H^{\mathcal{ND}} \bw_h)}{0}^2 \le C \eta_c\znorm{\bm{\nabla} \bw_h}{0}^2 \\
        & \le C \left( \eta_c \znorm{\bu_h}{0}^2 + \eta_c \znorm{\CURL \bu_h}{0}^2\right) \le C a_c(\bu_h, \bu_h).
        \end{aligned}
    \end{equation}
\end{itemize}
Hence, by combining \eqref{eq:est chi_0}, \eqref{eq:est w_0}, and \eqref{eq:est curl w_0}, the estimate \eqref{eq:ND coarse estimate} holds.
\end{proof}

We next consider an estimate for local components. Since the proof of the lemma is overlong, we estimate each term in \eqref{eq:ND stable decomposition 1} individually using propositions and put them together later.
\begin{prop}\label{prop:est chi_i}
    Consider $\chi_i$ defined in \eqref{eq:ND stable decomposition 2} and assume that $\eta_c \le 1$. Then, we have
    \[ 
        \sum_{i=1}^N a_c(\nabla \chi_i, \nabla \chi_i) \le C \left(\max_{1 \le i \le N}\left(1 + \dfrac{H_i}{\delta_i}\right)\right)a_c(\bu_h, \bu_h),
    \] 
    where the constant $C$ is independent from $N$, $h$, $H_i$, $\delta_i$, and $\eta_c$.
\end{prop}
\begin{proof}
We have the following estimate from \cite[Lemma 3.12]{MR2104179} and \eqref{eq:ND decomp est 1}:
\begin{equation}\label{eq:est chi_i}
    \begin{aligned}
    \sum_{i=1}^N a_c(\nabla \chi_i, \nabla \chi_i) & = C \left(\max_{1 \le i \le N}\left(1 + \dfrac{H_i}{\delta_i}\right)\right) \znorm{\nabla \chi_h}{0}^2 \\
    & \le C \left(\max_{1 \le i \le N}\left(1 + \dfrac{H_i}{\delta_i}\right)\right)\znorm{\bu_h}{0}^2\\
    & \le C \left(\max_{1 \le i \le N}\left(1 + \dfrac{H_i}{\delta_i}\right)\right)a_c(\bu_h, \bu_h).
    \end{aligned}
\end{equation}
\end{proof}

\begin{prop}\label{prop:est w_i}
    Consider $\bw_i$ defined in \eqref{eq:ND stable decomposition 2} and assume that $\eta_c \le 1$. Then, we have
    \[ 
        \sum_{i=1}^N a_c(\bw_i, \bw_i) \le C \left(\max_{1 \le i \le N}\left(1 + \dfrac{H_i}{\delta_i}\right)\right)a_c(\bu_h, \bu_h),
    \] 
    where the constant $C$ is independent from $N$, $h$, $H_i$, $\delta_i$, and $\eta_c$.
\end{prop}
\begin{proof}
    By using Lemma~\ref{lem:interpolation with theta_i}, the properties of $\theta_i$ in \eqref{eq:property theta_i}, the triangle inequality, \eqref{eq:cochain local stability}, the inverse estimate, the finite covering property in Assumption~\ref{assump:finite covering}, \eqref{eq:est w_0}, and \eqref{eq:ND decomp est 1}, we obtain 
    \begin{equation}\label{eq:est w_i}
    \begin{aligned}
        \sum_{i=1}^N \mnorm{\bw_i}{0}{\Omega_i'}^2 & \le C \sum_{i=1}^N \mnorm{\theta_i(\pi_{h}^{\mathcal{ND}}\bw_h - \bw_0)}{0}{\Omega_i'}^2 \\ 
        & \le C \left(\znorm{\bw_h}{0}^2 + \znorm{\bw_0}{0}^2 \right) \le C \znorm{\bw_h}{0}^2 \le C \znorm{\bu_h}{0}^2.
    \end{aligned}
    \end{equation}

    Let $\bv_1 = \pi_{h}^{\mathcal{ND}} \bw_h - \bw_h$, $\bv_2 = \bw_h - \bw_0 = \bw_h - Q_H^{\mathcal{ND}} \bw_h$, and $\bv = \bv_1 + \bv_2$. Due to Lemma~\ref{lem:interpolation with theta_i}, the construction of $\theta_i$, Assumption~\ref{assump:finite covering}, and the triangle inequality, we have

    \begin{equation}\label{eq:est curl w_i 1}
        \begin{aligned}
        & \sum_{i=1}^N \eta_c \mnorm{\CURL \bw_i}{0}{\Omega_i'}^2 \le C \eta_c \znorm{\CURL \bv}{0}^2 + C \eta_c \sum_{i=1}^N \delta_i^{-2} \mnorm{\bv}{0}{\Omega_{i, \delta}}^2 \\
        & \le C \eta_c \znorm{\CURL \bv}{0}^2 + C \eta_c \sum_{i=1}^N \delta_i^{-2} \mnorm{\bv_1}{0}{\Omega_{i, \delta}}^2 + C \eta_c \sum_{i=1}^N \delta_i^{-2} \mnorm{\bv_2}{0}{\Omega_{i, \delta}}^2\\
        & := E_1 + E_2 + E_3.
        \end{aligned}
    \end{equation}

    We first consider $E_1$. By using the triangle inequality, Lemma~\ref{lem:L2 projection}, \eqref{eq:commuting properties}, \eqref{eq:cochain local stability}, and \eqref{eq:ND decomp est 2}, we obtain 
    \begin{equation}\label{eq:est E1}
        \begin{aligned}
        E_1 & = C \eta_c \znorm{\CURL \bv}{0}^2 \le C \eta_c\znorm{\CURL \left(\pi_{h}^{\mathcal{ND}} \bw_h\right)}{0}^2 + C \eta_c\znorm{\CURL \left(Q_H^{\mathcal{ND}} \bw_h\right)}{0}^2 \\
        & \le C \eta_c\znorm{\pi_h^{\mathcal{RT}}\left(\CURL \bw_h\right)}{0}^2 + C \eta_c\znorm{\bm{\nabla} \bw_h}{0}^2 \\
        & \le C \eta_c \znorm{\CURL \bw_h}{0}^2 + C \eta_c \znorm{\bm{\nabla} \bw_h}{0}^2 \\
        & \le C \eta_c \znorm{\bm{\nabla} \bw_h}{0}^2 \le C  \left( \eta_c \znorm{\bu_h}{0}^2 + \eta_c \znorm{\CURL \bu_h}{0}^2 \right) \le C a_c(\bu_h, \bu_h).
        \end{aligned}
    \end{equation}

    Regarding $E_2$, the following estimate holds from an error estimate, the finite covering property in Assumption~\ref{assump:finite covering}, and \eqref{eq:ND decomp est 2}:
    \begin{equation}\label{eq:est E2}
        \begin{aligned}
        E_2 & = C \eta_c\sum_{i=1}^N \delta_i^{-2} \mnorm{\bv_1}{0}{\Omega_{i, \delta}} \le C \eta_c\sum_{i=1}^N h^{-2} \mnorm{\pi_{h}^{\mathcal{ND}} \bw_h - \bw_h}{0}{\Omega_i'}^2\\
        & \le C \eta_c\znorm{\bm{\nabla} \bw_h}{0}^2 \le C \left( \eta_c\znorm{\bu_h}{0}^2 + \eta_c\znorm{\CURL \bu_h}{0}^2 \right) \le C a_c(\bu_h, \bu_h).
        \end{aligned}
    \end{equation}

    We finally estimate $E_3$. Lemma~\ref{lem:est layer Omega_i} implies 
    \begin{equation}\label{eq:est E3}
        \begin{aligned}
        E_3 & = C \eta_c \sum_{i=1}^N \delta_i^{-2} \mnorm{\bv_2}{0}{\Omega_{i, \delta}}^2 \\
        & \le C \eta_c  \sum_{i=1}^N \sum_{j \in I_i} \sum_{\substack{K \in \cT_H, \\ K \subset \Omega_j}} \left(1 + \dfrac{H_i}{\delta_i}\right) \seminorm{\bv_2}{1}{K}^2 + C \eta_c \sum_{i=1}^N \sum_{j \in I_i} \sum_{\substack{K \in \cT_H, \\ K \subset \Omega_j}} \dfrac{1}{\delta_i H_i} \mnorm{\bv_2}{0}{K}^2 \\
        & := E_{3, 1} + E_{3, 2}.
        \end{aligned}
    \end{equation}
    We have the following estimate from the triangle inequality, the inverse estimate, and Lemma~\ref{lem:L2 projection K}:
    \begin{equation}\label{eq:est v_2 h_1}
        \begin{aligned}
            \seminorm{\bv_2}{1}{K} & = \seminorm{\bw_h - Q_H^{\mathcal{ND}} \bw_h}{1}{K} \le 
            \seminorm{\bw_h}{1}{K} + \seminorm{Q_H^{\mathcal{ND}} \bw_h}{1}{K} \\
            & = \seminorm{\bw_h}{1}{K} + \seminorm{Q_H^{\mathcal{ND}} \bw_h - Q_{H, K}^0\bw_h}{1}{K} \\
            & \le \seminorm{\bw_h}{1}{K} + C H^{-1} \mnorm{Q_H^{\mathcal{ND}} \bw_h - Q_{H, K}^0 \bw_h}{0}{K} \\
            & \le \seminorm{\bw_h}{1}{K} + C H^{-1} \mnorm{Q_H^{\mathcal{ND}} \bw_h - \bw_h}{0}{K} 
             + C  H^{-1} \mnorm{ \bw_h - Q_{H, K}^0\bw_h}{0}{K} \\
            & \le C \seminorm{\bw_h}{1}{K} + CH^{-1} \mnorm{Q_H^{\mathcal{ND}} \bw_h - \bw_h}{0}{K}.
        \end{aligned}
    \end{equation}

    Let $\Xi = \displaystyle{\max_{1 \le i \le n} \left(\dfrac{H_i}{\delta_i}\right)}$. Then, by using \eqref{eq:est v_2 h_1}, Lemma~\ref{lem:L2 projection} and \eqref{eq:ND decomp est 2}, we have
    \begin{equation}\label{eq:est E31}
        \begin{aligned}
            E_{3, 1} & = C \eta_c \sum_{i=1}^N \sum_{j \in I_i} \sum_{\substack{K \in \cT_H, \\ K \subset \Omega_j}} \left(1 + \dfrac{H_i}{\delta_i}\right) \seminorm{\bv_2}{1}{K}^2 \le C \eta_c (1 + \Xi)\sum_{i=1}^N \sum_{j \in I_i} \sum_{\substack{K \in \cT_H, \\ K \subset \Omega_j}} \seminorm{\bv_2}{1}{K}^2\\
            & \le \eta_c (1 + \Xi) \left(C \sum_{i=1}^N \sum_{j \in I_i} \sum_{\substack{K \in \cT_H, \\ K \subset \Omega_j}} \seminorm{\bw_h}{1}{K}^2 + C H^{-2}\sum_{i=1}^N \sum_{j \in I_i} \sum_{\substack{K \in \cT_H, \\ K \subset \Omega_j}} \mnorm{Q_H^{\mathcal{ND}} \bw_h - \bw_h}{0}{K}^2\right) \\
            & \le \eta_c (1 + \Xi) \left(C \znorm{\bm{\nabla} \bw_h}{0}^2 + C H^{-2} \znorm{Q_H^{\mathcal{ND}} \bw_h - \bw_h}{0}^2 \right)\\
            & \le C \eta_c(1 + \Xi) \znorm{\bm{\nabla} \bw_h}{0}^2 \le C (1 + \Xi)\left( \eta_c\znorm{\bu_h}{0}^2 + \eta_c \znorm{\CURL \bu_h}{0}^2 \right) \\
            & \le C(1 + \Xi)a_c(\bu_h, \bu_h).
        \end{aligned}
    \end{equation}
    Moreover, from Assumption~\ref{assump:H}, Lemma~\ref{lem:L2 projection}, and \eqref{eq:ND decomp est 2}, we obtain
    \begin{equation}\label{eq:est E32}
        \begin{aligned}
        E_{3, 2} & = C \eta_c \sum_{i=1}^N \sum_{j \in I_i} \sum_{\substack{K \in \cT_H, \\ K \subset \Omega_j}} \dfrac{1}{\delta_i H_i} \mnorm{\bv_2}{0}{K}^2 \\
        & \le C \eta_c H^{-2} \sum_{i=1}^N \sum_{j \in I_i} \sum_{\substack{K \in \cT_H, \\ K \subset \Omega_j}} \dfrac{H_i}{\delta_i} \mnorm{Q_H^{\mathcal{ND}} \bw_h - \bw_h}{0}{K}^2 \\
        & \le C \eta_c \Xi H^{-2} \znorm{Q_H^{\mathcal{ND}} \bw_h - \bw_h}{0}^2 \le C \eta_c \Xi \znorm{\bm{\nabla} \bw_h}{0}^2 \\
        & \le C \Xi \left( \eta_c \znorm{\bu_h}{0}^2 + \eta_c \znorm{\CURL \bu_h}{0}^2 \right) \le C \Xi a_c(\bu_h, \bu_h).
        \end{aligned}
    \end{equation}

    Thus, by using \eqref{eq:est curl w_i 1}, \eqref{eq:est E1}, \eqref{eq:est E2}, \eqref{eq:est E3}, \eqref{eq:est E31}, and \eqref{eq:est E32}, the following estimate holds:
    \begin{equation}\label{eq:est curl w_i final}
            \sum_{i=1}^N \eta_c \mnorm{\CURL \bw_i}{0}{\Omega_i'}^2 \le C \max_{1 \le i \le N}\left(1 + \dfrac{H_i}{\delta_i}\right) a_c(\bu_h, \bu_h).
    \end{equation}

    Combining \eqref{eq:est w_i} and \eqref{eq:est curl w_i final}, we have
    \[ 
        \sum_{i=1}^N a_c(\bw_i, \bw_i) \le C \left(\max_{1 \le i \le N}\left(1 + \dfrac{H_i}{\delta_i}\right)\right)a_c(\bu_h, \bu_h).
    \] 
\end{proof}

\begin{prop}\label{prop:est tu_i}
    Consider $\widetilde{\bu}_i$ defined in \eqref{eq:ND stable decomposition 2} and assume that $\eta_c \le 1$. Then, we have
    \[ 
        \sum_{i=1}^N a_c(\widetilde{\bu}_i, \widetilde{\bu}_i) \le a_c(\bu_h, \bu_h),
    \] 
    where the constant $C$ is independent from $N$, $h$, $H_i$, $\delta_i$, and $\eta_c$.
\end{prop}
\begin{proof}
From Lemma~\ref{lem:interpolation with theta_i}, \eqref{eq:property theta_i}, Assumption~\ref{assump:finite covering}, the inverse inequality, \eqref{eq:ND decomp est 1}, and \eqref{eq:ND decomp est 2}, we have
\begin{equation}\label{eq:est tu_i}
    \sum_{i=1}^N \mnorm{\widetilde{\bu}_i}{0}{\Omega_i'}^2 \le C\sum_{i=1}^N \mnorm{\theta_i \widetilde{\bu}_h}{0} {\Omega_i'}^2 \le C \znorm{\widetilde{\bu}_h}{0}^2 \le C \znorm{\bu_h}{0}^2
\end{equation}
and
\begin{equation}\label{eq:est curl tu_i}
    \begin{aligned}
        \sum_{i=1}^N \eta_c \mnorm{\CURL \widetilde{\bu}_i}{0}{\Omega_i'}^2 & \le C \eta_c \sum_{i=1}^N \delta_i^{-2} \mnorm{\widetilde{\bu}_h}{0}{\Omega_i'}^2 + C \eta_c \znorm{\CURL \widetilde{\bu}_h}{0}^2 \\
        & \le C \eta_c h^{-2} \znorm{\widetilde{\bu}_h}{0}^2 \le C \left( \eta_c \znorm{\bu_h}{0}^2 + \eta_c \znorm{\CURL \bu_h}{0}^2\right) \\
        & \le C a_c(\bu_h, \bu_h).
        \end{aligned}
\end{equation}
We therefore have
\[ 
    \sum_{i=1}^N a_c(\widetilde{\bu}_i, \widetilde{\bu}_i) \le C a_c(\bu_h, \bu_h).
\] 
\end{proof}

\begin{lem}\label{lem:ND local estimate}
    Assume that the constant $\eta_c$ in \eqref{eq:H(curl) bilinear form} is less than or equal to one. Then, we have the following estimate for the local components in \eqref{eq:ND stable decomposition 1}:
    \begin{equation}\label{eq:ND local estimate}
        \sum_{i=1}^N a_c(\bu_i, \bu_i) \le C \left(\max_{1 \le i \le N}\left(1 + \dfrac{H_i}{\delta_i}\right)\right)a_c(\bu_h, \bu_h),
    \end{equation}
    where the constant $C$ does not depend on $N$, $h$, $H_i$, $\delta_i$, and $\eta_c$.
\end{lem}
\begin{proof}
Based on Propositions~\ref{prop:est chi_i}, \ref{prop:est w_i}, and \ref{prop:est tu_i}, we have \eqref{eq:ND local estimate}.
\end{proof}

We finally have an estimate of the condition number for our $\HCurl$ model problem. 
\begin{thm}\label{thm:ND condition number estimate}
    Let $\eta_c \le 1$. We then have the following estimate:
    \begin{equation}\label{eq:cond est ND}
        \kappa(M_c^{-1}A_c) \le C\max_{1 \le i \le N}\left(1 + \dfrac{H_i}{\delta_i}\right),
    \end{equation}
    where the constant $C$ does not depend on the mesh sizes, $H_i$, $\delta_i$, $\eta_c$, and the number of subdomains but may depend on $N_0$.
\end{thm}
\begin{proof}
    We have \eqref{eq:cond est ND} from Lemmas~\ref{lem:smallest eigenvalue ND}, \ref{lem:largest eigenvalue}, \ref{lem:ND coarse estimate}, and \ref{lem:ND local estimate}.
\end{proof}

\begin{cor}
    If the first Betti number of the domain $\Omega$ vanishes, we have \eqref{eq:cond est ND} in Theorem~\ref{thm:ND condition number estimate} without the assumption $\eta_c \le 1$.
\end{cor}
\begin{proof}
    If the first Betti number of the domain $\Omega$ vanishes, we have a more favorable bound in \eqref{eq:ND decomp est 2}, i.e., the right hand side can be replaced by the curl term only. Thus, we can have \eqref{eq:cond est ND} in Theorem~\ref{thm:ND condition number estimate} without the assumption $\eta_c \le 1$. For more detail, see \cite[Section 5.1]{HP:2019:RegularDecomp}.
\end{proof}
\subsection{Condition number estimate for $\HDiv$} Like the $\HCurl$ case, we consider the following decomposition for any $\bp_h \in \RT$ based on Lemma~\ref{lem:HP decomp RT}:
\begin{equation}\label{eq:RT stable decomposition 1}
    \begin{aligned}
        \bp_0 & := \CURL \bm{\rho}_0 + \br_0, \\
        \bp_i & := \CURL \bm{\rho}_i + \br_i + \widetilde{\bp}_i, \quad i = 1, 2, \cdots, N, 
    \end{aligned}
\end{equation}
where
\begin{equation}\label{eq:RT stable decomposition 2}
    \begin{aligned}
        \bm{\rho}_0 & = Q_H^{\mathcal{ND}} \bm{\sigma}_h, \\
        \br_0 & = Q_H^{\mathcal{RT}} \br_h, \\
        \bm{\rho}_i & = \Pi_h^{\mathcal{ND}} (\theta_i(\pi_h^{\mathcal{ND}}\bm{\sigma}_h - \bm{\rho}_0)) + \Pi_h^{\mathcal{ND}}\left(\theta_i \widetilde{\bm{\rho}}_h \right), \\
        \br_i & = \Pi_{h}^{\mathcal{RT}}(\theta_i(\pi_{h}^{\mathcal{RT}} \br_h - \br_0)), \\
        \widetilde{\bp}_i & = \Pi_{h}^{\mathcal{RT}}(\theta_i \widetilde{\bp}_h).
    \end{aligned}
\end{equation}
Here, $\bm{\rho}_h = \nabla \mu_h + \Pi_h^{\mathcal{ND}} \bm{\sigma}_h + \widetilde{\bm{\rho}}_h$ is given based on Lemma~\ref{lem:HP decomp ND}. We note that the operators $\Pi_h^{\mathcal{ND}}$ and $\Pi_h^{\mathcal{RT}}$ are introduced in Section~\ref{sec:model problems and finite elements}. We also remark that the partition of unity set $\left\{ \theta_i \right\}$ is constructed in Section~\ref{sec:overlapping schwarz methods} and the $L^2-$projection operators are defined in Section~\ref{subsec:preliminaries}. In addition, the cochain projections $\pi_h^{\mathcal{ND}}$ and $\pi_h^{\mathcal{RT}}$ are introduced in Section~\ref{subsec:regular decomposition}. Obviously, we have $\bp_0 \in \RTH$, $\bp_i \in \RT^{(i)}$, and $\bp_h = \sum_{i=0}^N \bp_i$. Similarly, we consider estimates for the coarse and the local components.

We first consider the stability of $\bp_0$ in \eqref{eq:RT stable decomposition 1}. We recall that we have the constant $\eta_d$ in the bilinear form \eqref{eq:H(div) bilinear form}.
\begin{lem}\label{lem:RT coarse estimate}
    Assume that the constant $\eta_d$ in \eqref{eq:H(div) bilinear form} is less than or equal to one. Then, we have the following estimate for the coarse component $\bp_0$ in \eqref{eq:RT stable decomposition 1}:
    \begin{equation}\label{eq:RT coarse estimate}
        a_d(\bp_0, \bp_0) \le C a_d(\bp_h, \bp_h),
    \end{equation}
    where the constant $C$ does not depend on $N$, $h$, $H_i$, $\delta_i$, and $\eta_d$.
\end{lem}
\begin{proof}
    Based on the decomposition $\bp_0 = \CURL \bm{\rho}_0 + \br_0$, we consider each term one by one.
    \begin{itemize}
        \item Term $\bm{\rho}_0$:
        
        With a similar argument to \eqref{eq:est curl w_0} in Lemma~\ref{lem:ND coarse estimate}, \eqref{eq:ND decomp est 2}, and \eqref{eq:RT decomp est 1}, we have
        \begin{equation}\label{eq:est curl rho_0}
            \begin{aligned}
                a_d(\CURL \bm{\rho}_0, \CURL \bm{\rho}_0) & = \znorm{\CURL \bm{\rho}_0}{0}^2 \le C \znorm{\bp_h}{0}^2.
            \end{aligned}
        \end{equation}

        \item Term $\br_0$:
        
        From the projection property and \eqref{eq:RT decomp est 1}, we obtain
        \begin{equation}\label{eq:est r_0}
            \znorm{\br_0}{0}^2 = \znorm{Q_H^{\mathcal{RT}} \br_h}{0}^2 \le \znorm{\br_h}{0}^2 \le C \znorm{\bp_h}{0}^2.
        \end{equation}

        By using Lemma~\ref{lem:L2 projection} and \eqref{eq:RT decomp est 2}, the following estimate holds:
        \begin{equation}\label{eq:est div r_0}
            \begin{aligned}
            \eta_d \znorm{\DIV \br_0}{0}^2 & = \eta_d \znorm{\DIV \left(Q_H^{\mathcal{RT}} \br_h\right)}{0}^2 \le C \eta_d \znorm{\bm{\nabla} \br_h}{0}^2 \\
            & \le C \left(\eta_d \znorm{\DIV \bp_h}{0}^2 + \eta_d\znorm{\bp_h}{0}^2\right) \le C a_d(\bp_h, \bp_h).
            \end{aligned}
        \end{equation}
    \end{itemize}
    We therefore have \eqref{eq:RT coarse estimate} from \eqref{eq:est curl rho_0}, \eqref{eq:est r_0}, and \eqref{eq:est div r_0}.
\end{proof}

We next consider three propositions to estimate each term associated with the local components in \eqref{eq:RT stable decomposition 1} separately. 

\begin{prop}\label{prop:est rho_i}
Assume that $\eta_d \le 1$. Then, the term $\rho_i$ introduced in \eqref{eq:RT stable decomposition 2} has the following estimate:
\[ 
    \sum_{i=1}^N a_d(\CURL \bm{\rho}_i, \CURL \bm{\rho}_i) \le C \max_{1 \le i \le N}\left( 1 + \dfrac{H_i}{\delta_i}\right) a_d(\bp_h, \bp_h),
\] 
where the constant $C$ is independent from $N$, $h$, $H_i$, $\delta_i$, and $\eta_d$.
\end{prop}
\begin{proof}
We can use the same methods in Propositions~\ref{prop:est w_i} and \ref{prop:est tu_i} and \eqref{eq:RT decomp est 1}. We then have
\[ 
    \begin{aligned}
    \sum_{i=1}^N a_d(\CURL \bm{\rho}_i, \CURL \bm{\rho}_i) & = \sum_{i=1}^N \mnorm{\CURL \bm{\rho}_i}{0}{\Omega_i'}^2 \le C \max_{1 \le i \le N}\left(1 + \dfrac{H_i}{\delta_i}\right)\znorm{\bp_h}{0}^2 \\
    & \le C \max_{1 \le i \le N}\left(1 + \dfrac{H_i}{\delta_i}\right) a_d(\bp_h, \bp_h).
    \end{aligned}
\]  
\end{proof}

\begin{prop}\label{prop:est r_i}
Assume that $\eta_d \le 1$. Then, the term $\br_i$ introduced in \eqref{eq:RT stable decomposition 2} has the following estimate:
\[ 
    \sum_{i=1}^N a_d(\br_i, \br_i) \le C \max_{1 \le i \le N}\left( 1 + \dfrac{H_i}{\delta_i}\right) a_d(\bp_h, \bp_h),
\] 
where the constant $C$ is independent from $N$, $h$, $H_i$, $\delta_i$, and $\eta_d$.
\end{prop}
\begin{proof}
    With the same process with \eqref{eq:est w_i}, we obtain
    \begin{equation}\label{eq:est r_i}
        \begin{aligned}
            \sum_{i=1}^N \mnorm{\br_i}{0}{\Omega_i'}^2 
            & \le C \znorm{\br_h}{0}^2 \le C \znorm{\bp_h}{0}^2.
        \end{aligned}
    \end{equation}
    Let $\bq = \bq_1 + \bq_2$ with $\bq_1 = \pi_h^{\mathcal{RT}}\br_h - \br_h$ and $\bq_2 = \br_h - \br_0 = \br_h - Q_H^{\mathcal{RT}} \br_h$. By using a similar argument to \eqref{eq:est curl w_i 1}, we have
    \begin{equation}\label{eq:est div r_i 1}
        \begin{aligned}
        & \sum_{i=1}^N \eta_d \mnorm{\DIV \br_i}{0}{\Omega_i'}^2 \\
        & \le C \eta_d \znorm{\DIV \bq}{0}^2 + C \eta_d \sum_{i=1}^N \delta_i^{-2} \mnorm{\bq_1}{0}{\Omega_{i, \delta}}^2 + C \eta_d \sum_{i=1}^N \delta_i^{-2} \mnorm{\bq_2}{0}{\Omega_{i, \delta}}^2\\
        & := F_1 + F_2 + F_3.
        \end{aligned}
    \end{equation}

    Regarding $F_1$, we obtain the following estimate in a similar way to \eqref{eq:est E1}:
    \begin{equation}\label{eq:est F1}
        \begin{aligned}
        F_1 & = C \eta_d \znorm{\DIV \bq}{0}^2 \le C \eta_d\znorm{\DIV \left(\pi_{h}^{\mathcal{RT}} \br_h\right)}{0}^2 + C \eta_d\znorm{\DIV \left(Q_H^{\mathcal{RT}} \br_h\right)}{0}^2 \\
        & \le C \eta_d\znorm{\pi_h^0 \left(\DIV \br_h\right)}{0}^2 + C \eta_d\znorm{\bm{\nabla} \br_h}{0}^2 
        \le C \eta_d \znorm{\DIV \br_h}{0}^2 + C \eta_d \znorm{\bm{\nabla} \br_h}{0}^2 \\
        & \le C \eta_d \znorm{\bm{\nabla} \br_h}{0}^2 \le C \left(\eta_d \znorm{\DIV \bp_h}{0}^2 + \eta_d\znorm{\bp_h}{0}^2\right) \le C a_d(\bp_h, \bp_h).
        \end{aligned}
    \end{equation}

    We next estimate $F_2$. The following bound can be found using the argument in \eqref{eq:est E2}:
    \begin{equation}\label{eq:est F2}
        \begin{aligned}
        F_2 & = C \eta_d\sum_{i=1}^N \delta_i^{-2} \mnorm{\bq_1}{0}{\Omega_{i, \delta}} \le C \eta_d\sum_{i=1}^N h^{-2} \mnorm{\pi_{h}^{\mathcal{RT}} \br_h - \br_h}{0}{\Omega_i'}^2\\
        & \le C \eta_d\znorm{\bm{\nabla} \br_h}{0}^2 \le C \left(\eta_d \znorm{\DIV \bp_h}{0}^2 + \eta_d\znorm{\bp_h}{0}^2\right) \le C a_d(\bp_h, \bp_h).
        \end{aligned}
    \end{equation}

    Finally, we consider $F_3$. Like \eqref{eq:est E3}, from Lemma~\ref{lem:est layer Omega_i}, we have
    \begin{equation}\label{eq:est F3}
        \begin{aligned}
        F_3 & = C \eta_d \sum_{i=1}^N \delta_i^{-2} \mnorm{\bq_2}{0}{\Omega_{i, \delta}}^2 \\
        & \le C \eta_d  \sum_{i=1}^N \sum_{j \in I_i} \sum_{\substack{K \in \cT_H, \\ K \subset \Omega_j}} \left(1 + \dfrac{H_i}{\delta_i}\right) \seminorm{\bq_2}{1}{K}^2 + C \eta_d \sum_{i=1}^N \sum_{j \in I_i} \sum_{\substack{K \in \cT_H, \\ K \subset \Omega_j}} \dfrac{1}{\delta_i H_i} \mnorm{\bq_2}{0}{K}^2 \\
        & := F_{3, 1} + F_{3, 2}.
        \end{aligned}
    \end{equation}
    In the same way as \eqref{eq:est v_2 h_1}, we obtain
    \begin{equation}\label{eq:est q_2 h_1}
        \seminorm{\bq_2}{1}{K} \le C \seminorm{\br_h}{1}{K} + CH^{-1} \mnorm{Q_H^{\mathcal{RT}}\br_h - \br_h}{0}{K}.
    \end{equation}
    Hence, in a similar way to \eqref{eq:est E31}, we have 
    \begin{equation}\label{eq:est F31}
        \begin{aligned}
            F_{3, 1} & = C \eta_d \sum_{i=1}^N \sum_{j \in I_i} \sum_{\substack{K \in \cT_H, \\ K \subset \Omega_j}} \left(1 + \dfrac{H_i}{\delta_i}\right) \seminorm{\bq_2}{1}{K}^2 \le C \eta_d (1 + \Xi)\sum_{i=1}^N \sum_{j \in I_i} \sum_{\substack{K \in \cT_H, \\ K \subset \Omega_j}} \seminorm{\bq_2}{1}{K}^2\\
            & \le \eta_d (1 + \Xi) \left(C \znorm{\bm{\nabla} \br_h}{0}^2 + C H^{-2} \znorm{Q_H^{\mathcal{RT}} \br_h - \br_h}{0}^2 \right)
            \le C \eta_d(1 + \Xi) \znorm{\bm{\nabla} \br_h}{0}^2 \\
            & \le C (1 + \Xi)\left(\eta_d \znorm{\DIV \bp_h}{0}^2 + \eta_d\znorm{\bp_h}{0}^2\right) \le C (1 + \Xi)a_d(\bp_h, \bp_h),
        \end{aligned}
    \end{equation}
    where $\Xi = \displaystyle{\max_{1 \le i \le N}\left(\dfrac{H_i}{\delta_i}\right)}$.

    In addition, the argument in \eqref{eq:est E32} gives
    \begin{equation}\label{eq:est F32}
        \begin{aligned}
        F_{3, 2} & = C \eta_d \sum_{i=1}^N \sum_{j \in I_i} \sum_{\substack{K \in \cT_H, \\ K \subset \Omega_j}} \dfrac{1}{\delta_i H_i} \mnorm{\bq_2}{0}{\Omega_i}^2 \le C \eta_d \Xi H^{-2} \znorm{Q_H^{\mathcal{RT}} \br_h - \br_h}{0}^2 \\
        & \le C \eta_d \Xi \znorm{\bm{\nabla} \br_h}{0}^2 \le C \Xi \left(\eta_d \znorm{\DIV \bp_h}{0}^2 + \eta_d\znorm{\bp_h}{0}^2\right) \le C \Xi a_d(\bp_h, \bp_h).
        \end{aligned}
    \end{equation}

    We therefore have the following inequality by using \eqref{eq:est div r_i 1}, \eqref{eq:est F1}, \eqref{eq:est F2}, \eqref{eq:est F3}, \eqref{eq:est F31}, and \eqref{eq:est F32}:
    \begin{equation}\label{eq:est div r_i final}
            \sum_{i=1}^N \eta_d \mnorm{\DIV \br_i}{0}{\Omega_i'}^2 \le C \max_{1 \le i \le N}\left(1 + \dfrac{H_i}{\delta_i}\right) a_d(\bp_h, \bp_h).
    \end{equation}
    Due to \eqref{eq:est r_i} and \eqref{eq:est div r_i final}, we finally have
    \begin{equation}
        \sum_{i=1}^N a_d(\br_i, \br_i) \le C \max_{1 \le i \le N}\left( 1 + \dfrac{H_i}{\delta_i}\right) a_d(\bp_h, \bp_h).
    \end{equation}
\end{proof}

\begin{prop}\label{prop:est tp_i}
Assume that $\eta_d \le 1$. Then, the term $\widetilde{\bp}_i$ introduced in \eqref{eq:RT stable decomposition 2} has the following estimate:
\[ 
    \sum_{i=1}^N a_d(\widetilde{\bp}_i, \widetilde{\bp}_i) \le C a_d(\bp_h, \bp_h),
\] 
where the constant $C$ is independent from $N$, $h$, $H_i$, $\delta_i$, and $\eta_d$.
\end{prop}
\begin{proof}
    By using Lemma~\ref{lem:interpolation with theta_i}, the construction of the partition of unity set $\left\{\theta_i \right\}$ in \eqref{eq:property theta_i}, and \eqref{eq:RT decomp est 1}, we obtain
    \begin{equation}\label{eq:est tp_i}
        \sum_{i=1}^N \mnorm{\widetilde{\bp}_i}{0}{\Omega_i'}^2 \le C \sum_{i=1}^N \mnorm{\theta_i \widetilde{\bp}_h}{0}{\Omega_i'}^2 \le C \znorm{\widetilde{\bp}_h}{0}^2 \le C \znorm{\bp_h}{0}^2.
    \end{equation}
    From  \eqref{eq:property theta_i}, Lemma~\ref{lem:interpolation with theta_i}, the inverse inequality, and \eqref{eq:RT decomp est 2}, we have
    \begin{equation}\label{eq:est div tp_i}
        \begin{aligned}
            \eta_d \sum_{i=1}^N \mnorm{\DIV \widetilde{\bp}_i}{0}{\Omega_i'}^2 & \le C \eta_d \sum_{i=1}^N \delta_i^{-2} \mnorm{\widetilde{\bp}_h}{0}{\Omega_i'}^2 + C \eta_d \znorm{\DIV \widetilde{\bp}_h}{0}^2 \\
            & \le C \eta_d h^{-2} \znorm{\widetilde{\bp}_h}{0}^2 \le C \left(\eta_d \znorm{\DIV \bp_h}{0}^2 + \eta_d\znorm{\bp_h}{0}^2\right) \\
            & \le C a_d(\bp_h, \bp_h).
        \end{aligned}
    \end{equation}
    We therefore have
    \[ 
        \sum_{i=1}^N a_d(\widetilde{\bp}_i, \widetilde{\bp}_i) \le C a_d(\bp_h, \bp_h).
    \] 
\end{proof}
\begin{lem}\label{lem:RT local estimate}
    Assume that the constant $\eta_d$ in \eqref{eq:H(div) bilinear form} is less than or equal to one. Then, we have the following estimate for the local components in \eqref{eq:ND stable decomposition 1}:
    \begin{equation}\label{eq:RT local estimate}
        \sum_{i=1}^N a_d(\bp_i, \bp_i) \le C \left(\max_{1 \le i \le N}\left(1 + \dfrac{H_i}{\delta_i}\right)\right)a_d(\bp_h, \bp_h),
    \end{equation}
    where the constant $C$ does not depend on $N$, $h$, $H_i$, $\delta_i$, and $\eta_d$.
\end{lem}
\begin{proof}
Based on Propositions~\ref{prop:est rho_i}, \ref{prop:est r_i}, and \ref{prop:est tp_i}, we have \eqref{eq:RT local estimate}.
\end{proof}

Finally, We obtain an estimate of the condition number for our $\HDiv$ model problem. 
\begin{thm}\label{thm:RT condition number estimate}
    Let $\eta_d \le 1$. We then have the following estimate:
    \begin{equation}\label{eq:cond est RT}
        \kappa(M_d^{-1}A_d) \le C\max_{1 \le i \le N}\left(1 + \dfrac{H_i}{\delta_i}\right),
    \end{equation}
    where the constant $C$ does not depend on the mesh sizes, $H_i$, $\delta_i$, $\eta_d$, and the number of subdomains but may depend on $N_0$.
\end{thm}
\begin{proof}
    We obtain \eqref{eq:cond est RT} by using Lemmas~\ref{lem:smallest eigenvalue RT}, \ref{lem:largest eigenvalue}, \ref{lem:RT coarse estimate}, and \ref{lem:RT local estimate}.
\end{proof}

\begin{cor}
    If the second Betti number of the domain $\Omega$ is zero, we have \eqref{eq:cond est RT} in Theorem~\ref{thm:RT condition number estimate} without the assumption $\eta_d \le 1$.
\end{cor}
\begin{proof}
Provided that the second Betti number of the domain $\Omega$ is zero, we can replace the upper bound of \eqref{eq:RT decomp est 2} by simply the divergence term. We therefore remove the assumption regarding the coefficient $\eta_d$ in Theorem~\ref{thm:RT condition number estimate}; see \cite[Section 5.2]{HP:2019:RegularDecomp} for more detail.
\end{proof}

\section{Numerical experiments}\label{sec:numerical experiments}

In this section, we perform three numerical experiments for $\HCurl$ problems. In each experiment, we report the error profile associated with the following errors:
\begin{equation*}
    \begin{aligned}
        \t{Error 1 } & := \| \Pi_h^{\mathcal{ND}} \b u -  \b u_h \|_0 \\
        \t{Error 2 } & := \|\t{curl}(\Pi_h^{\mathcal{ND}} \b u -  \b u_h) \|_0 \t{ for } 2D
                      \t{ or } \|\bb{curl}(\Pi_h^{\mathcal{ND}} \b u -  \b u_h) \|_0 \t{ for } 3D.
    \end{aligned}
\end{equation*}
In this work,  we proved that
\an{\label{b-d}  C_{\t{\tiny low}} \frac{a_c(\bu_h,\bu_h)}{\left(1+H/\delta\right)}
                  \le a_c(M_c^{-1}A_c\bu_h,\bu_h)
               \le C_{\t{\tiny high}}  a_c(\bu_h,\bu_h) }
  for some positive constants $C_{\t{\tiny low}}$ and $C_{\t{\tiny high}}$
   independent of $h$, $\delta$, and $H$. We also report the constants $C_{\t{\tiny low}}$ and $C_{\t{\tiny high}}$ obtained numerically in each example.

   \subsection{N\'{e}d\'{e}lec Type-1 rectangular element}
We solve the following $\bb{curl}\, \t{curl}$ equations on two domains:
\an{\label{e1} \ad{ \bb{curl} \, \t{curl}\, \b u_h + \b u_h &= \b f \quad &&\t{in } \Omega, \\
                    \b u_h\times \b n &= g   \quad &&\t{on } \partial \Omega, } }
where $\b n$ is the unit outer normal vector, and 
\a{ \Omega=\left(0,1\right)^2, \qquad \t{ or } \  \left(0,1\right)^2\setminus\left\{1/2 \right\}\times\left(0,1/2\right]. }
In both cases, the exact solution of \eqref{e1} is chosen as
\an{\label{s1} \b u=\p{y^5 \\ x^4}.  }
In both cases, the meshes used in the computation
    are uniform square meshes,  as shown in Figure \ref{g-sq}.
The results are listed in Table \ref{t1}, where we can see that the finite element
  solution converges at the optimal order in both norms on both domains.

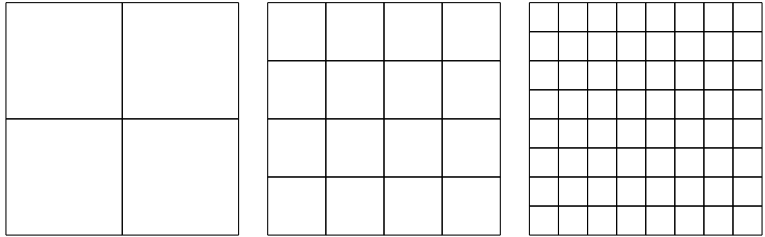
\begin{figure}[ht!]
 \begin{center} \setlength\unitlength{1.1pt}
\begin{picture}(260,80)(0,0)
  \def\tr{\begin{picture}(20,20)(0,0)
    \put(0,0){\line(1,0){20}}\put(0,20){\line(1,0){20}}
          \put(0,0){\line(0,1){20}} \put(20,0){\line(0,1){20}}   \end{picture}}
 
  {\setlength\unitlength{2.2pt}
 \multiput(0,0)(20,0){2}{\multiput(0,0)(0,20){2}{\tr}}}

  \multiput(90,0)(20,0){4}{\multiput(0,0)(0,20){4}{\tr}}

  {\setlength\unitlength{0.55pt}
 \multiput(360,0)(20,0){8}{\multiput(0,0)(0,20){8}{\tr}}}

 \end{picture}\end{center}
\caption{The first three grids for computing Tables \ref{t1}--\ref{t2}.}
\label{g-sq}
\end{figure}

\begin{table}[ht]
    \centering   \renewcommand{\arraystretch}{1.05}
    \caption{ Error profile for \eqref{s1}  on grids as shown in Figure \ref{g-sq}.
     }\label{t1} 
  \begin{tabular}{c|cc|cc|cc|cc}
  \hline
  Grid & Error 1 & Order   
        & Error 2 & Order & Error 1 & Order   
        & Error 2 & Order  \\
  \hline
   &\multicolumn{4}{c|}{On $\Omega=(0,1)^2$.}  & \multicolumn{4}{c}{On $\Omega=
   (0,1)^2\setminus\{\frac 12\}\times(0,\frac 12]$.} \\ \hline
   1&   8.97E-2 &  0.0&   3.11E-1 &  0.0 &   1.27E-1 &  0.0&   2.77E-1 &  0.0 \\
   2&   2.67E-2 &  1.7&   8.80E-2 &  1.8 &   3.36E-2 &  1.9&   8.19E-2 &  1.8 \\ 
   3&   6.94E-3 &  1.9&   2.26E-2 &  2.0 &   8.49E-3 &  2.0&   2.13E-2 &  1.9 \\
   4&   1.75E-3 &  2.0&   5.69E-3 &  2.0 &   2.12E-3 &  2.0&   5.37E-3 &  2.0 \\
   5&   4.39E-4 &  2.0&   1.43E-3 &  2.0 &   5.29E-4 &  2.0&   1.35E-3 &  2.0 \\
   6&   1.10E-4 &  2.0&   3.56E-4 &  2.0 &   1.32E-4 &  2.0&   3.37E-4 &  2.0 \\
   7&   2.75E-5 &  2.0&   8.91E-5 &  2.0 &   3.30E-5 &  2.0&   8.43E-5 &  2.0 \\
   8&   6.87E-6 &  2.0&   2.23E-5 &  2.0 &   8.24E-6 &  2.0&   2.11E-5 &  2.0 \\
  \hline
  \end{tabular}%
  \end{table}%

In this example, we subdivide both $\Omega$ into four subdomains, as shown in
    Figure \ref{d2}.  
We consider a minimum overlapping domain decomposition. 
We plot the nodes of the third-level finite element function inside each subdomain.
We note that the horizontal nodes belong to the first component of the vector
   $H(\t{curl})$ function.
The difference between the graphs is at the nodes on the lower middle vertical edge,
   which is a boundary edge.

\begin{figure}[htb]\begin{center}  
\begin{picture}(320,160)(0,0) 
\put(0,-80){\includegraphics[width=180pt]{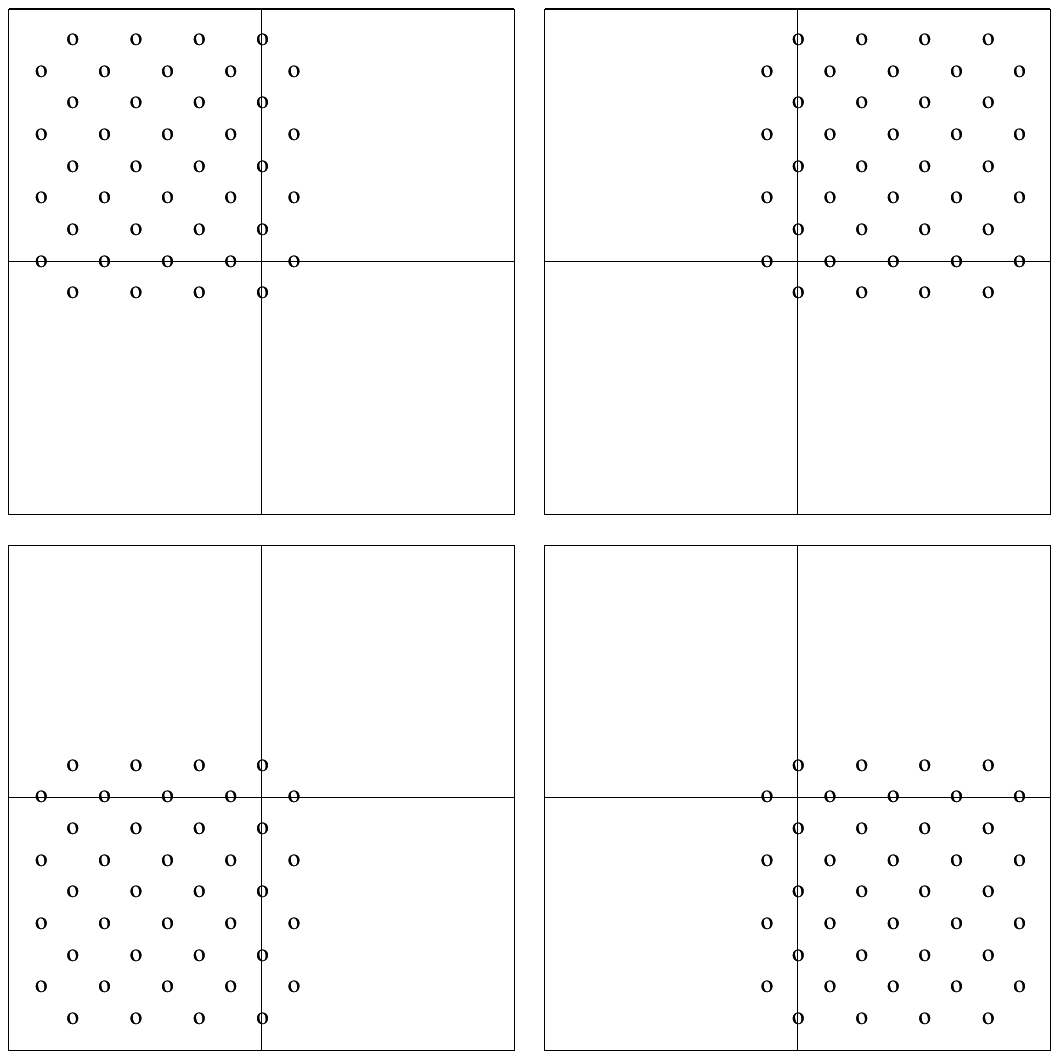}}
\put(160,-80){\includegraphics[width=180pt]{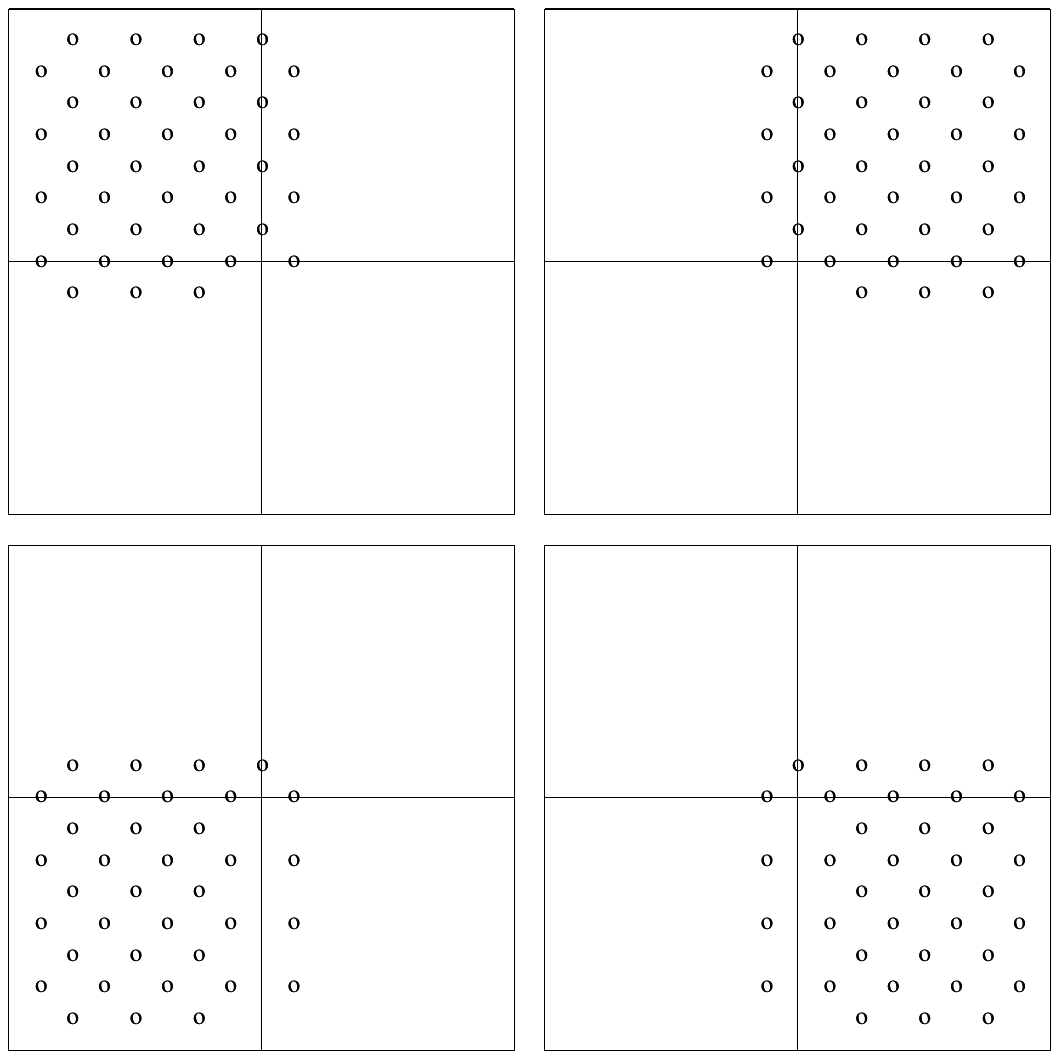}}
\end{picture}
\caption{Left: The level three function nodes in the four subdomains, where $\Omega=(0,1)^2$;
  Right: The 4-subdomain nodes for $\Omega=
   (0,1)^2\setminus\{\frac 12\}\times(0,\frac 12]$. }
\label{d2}
\end{center}
\end{figure}

We list the computer found constants of \eqref{b-d} in Table \ref{t2}.
As proved in the theory,  the constants remain bounded when doing
  domain decomposition methods on the non-convex domain  $\Omega=
(0,1)^2\setminus\{\frac 12\}\times(0,\frac 12]$.

\begin{table}[ht!]
    \centering   \renewcommand{\arraystretch}{1.1}
    \caption{The bounds for 4-subdomain small-overlap DD shown as in Figure \ref{d2}.
     }\label{t2} 
  \begin{tabular}{c|cc|cc}
  \hline
  Grid & $C_{\t{\tiny low}}$ in \eqref{b-d}&$C_{\t{\tiny high}}$ in \eqref{b-d} & $C_{\t{\tiny low}}$ in \eqref{b-d}&$C_{\t{\tiny high}}$ in \eqref{b-d} \\
  \hline
   &\multicolumn{2}{c|}{On $\Omega=(0,1)^2$.} &\multicolumn{2}{c}{On $\Omega=
   (0,1)^2\setminus\{\frac 12\}\times(0,\frac 12]$.} \\ \hline 
   2 &           2.015473 &           4.485408 &           1.958618 &           4.349539 \\
   3 &           2.553089 &           4.162126 &           2.271621 &           4.092826 \\
   4 &           2.843341 &           4.044754 &           2.414993 &           4.022131 \\
   5 &           2.999805 &           4.011540 &           2.463830 &           4.005278 \\
   6 &           3.088654 &           4.002915 &           2.477246 &           4.001282 \\
  \hline
  \end{tabular}%
  \end{table}%

  \subsection{Triangular N\'{e}d\'{e}lec element}
We solve the $\bb{curl}\t{curl}$ equation \eqref{e1} again on the two domains
\a{ \Omega=\left(0,1\right)^2, \qquad \t{ or } \  \left(0,1\right)^2\setminus\left\{1/2\right\}\times\left(0,1/2\right]. }
The exact solution of \eqref{e1} is chosen as
\an{\label{s2} \b u=\p{ x^2y^2 \\ x^2y}.  }
In both cases, the meshes used in the computation
    are uniform triangular meshes,  as shown in Figure \ref{g-tr}.
The results are listed in Table \ref{t21}, where we can see that the finite element
  solution converges at the optimal order in both norms on both domains.

\begin{figure}[ht!]
 \begin{center} \setlength\unitlength{1.1pt}
\begin{picture}(260,80)(0,0)
  \def\tr{\begin{picture}(20,20)(0,0)
    \put(0,0){\line(1,0){20}}\put(0,20){\line(1,0){20}}\put(0,20){\line(1,-1){20}}
          \put(0,0){\line(0,1){20}} \put(20,0){\line(0,1){20}}   \end{picture}}
 
  {\setlength\unitlength{2.2pt}
 \multiput(0,0)(20,0){2}{\multiput(0,0)(0,20){2}{\tr}}}

  \multiput(90,0)(20,0){4}{\multiput(0,0)(0,20){4}{\tr}}

  {\setlength\unitlength{0.55pt}
 \multiput(360,0)(20,0){8}{\multiput(0,0)(0,20){8}{\tr}}}

 \end{picture}\end{center}
\caption{The first three grids for computing Tables \ref{t21}--\ref{t22}.}
\label{g-tr}
\end{figure}
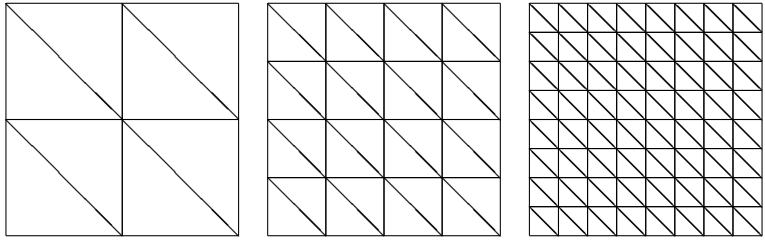

\begin{table}[ht]
    \centering   
    \caption{ Error profile for \eqref{s2}  on grids as shown in Figure \ref{g-tr}.
     }\label{t21} 
  \begin{tabular}{c|cc|cc|cc|cc}
  \hline
  Grid & Error 1 & Order   
        & Error 2 & Order & Error 1 & Order   
        & Error 2 & Order  \\
  \hline
   &\multicolumn{4}{c|}{On $\Omega=(0,1)^2$.}  &\multicolumn{4}{c}{On $\Omega=
   (0,1)^2\setminus\{\frac 12\}\times(0,\frac 12]$.} \\ \hline
   1&   1.17E-2 &  0.0&   7.19E-2 &  0.0 &   1.19E-2 &  0.0&   7.18E-2 &  0.0 \\
   2&   3.19E-3 &  1.9&   4.00E-2 &  0.8 &   3.29E-3 &  1.9&   4.00E-2 &  0.8 \\
   3&   8.12E-4 &  2.0&   2.05E-2 &  1.0 &   8.35E-4 &  2.0&   2.05E-2 &  1.0 \\
   4&   2.04E-4 &  2.0&   1.03E-2 &  1.0 &   2.10E-4 &  2.0&   1.03E-2 &  1.0 \\
   5&   5.10E-5 &  2.0&   5.15E-3 &  1.0 &   5.24E-5 &  2.0&   5.15E-3 &  1.0 \\
   6&   1.27E-5 &  2.0&   2.58E-3 &  1.0 &   1.31E-5 &  2.0&   2.58E-3 &  1.0 \\
   7&   3.19E-6 &  2.0&   1.29E-3 &  1.0 &   3.28E-6 &  2.0&   1.29E-3 &  1.0 \\
   8&   7.97E-7 &  2.0&   6.44E-4 &  1.0 &   8.19E-7 &  2.0&   6.44E-4 &  1.0 \\
  \hline
  \end{tabular}%
  \end{table}%

Again we do iterations based on domain decomposition methods with four subdomains for both domains $\Omega$, as shown in
    Figure \ref{d22}.   
We plot the nodes of the third-level finite element function inside each subdomain,
   in Figure \ref{d22}. 
The difference between the two graphs is at the nodes on the lower middle vertical edge,
   which is a boundary edge.

\begin{figure}[htb]\begin{center}  
\begin{picture}(320,160)(0,0) 
\put(0,-80){\includegraphics[width=180pt]{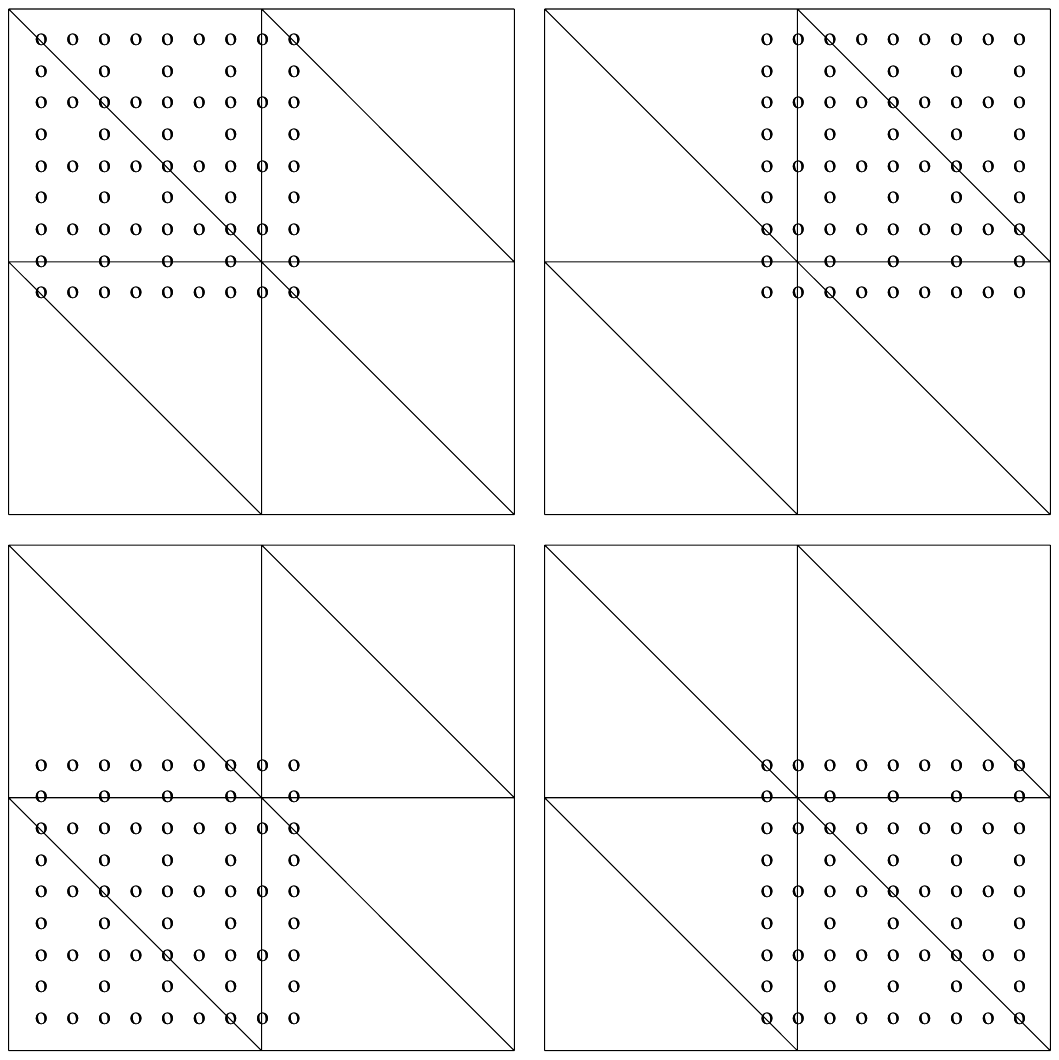}}
\put(160,-80){\includegraphics[width=180pt]{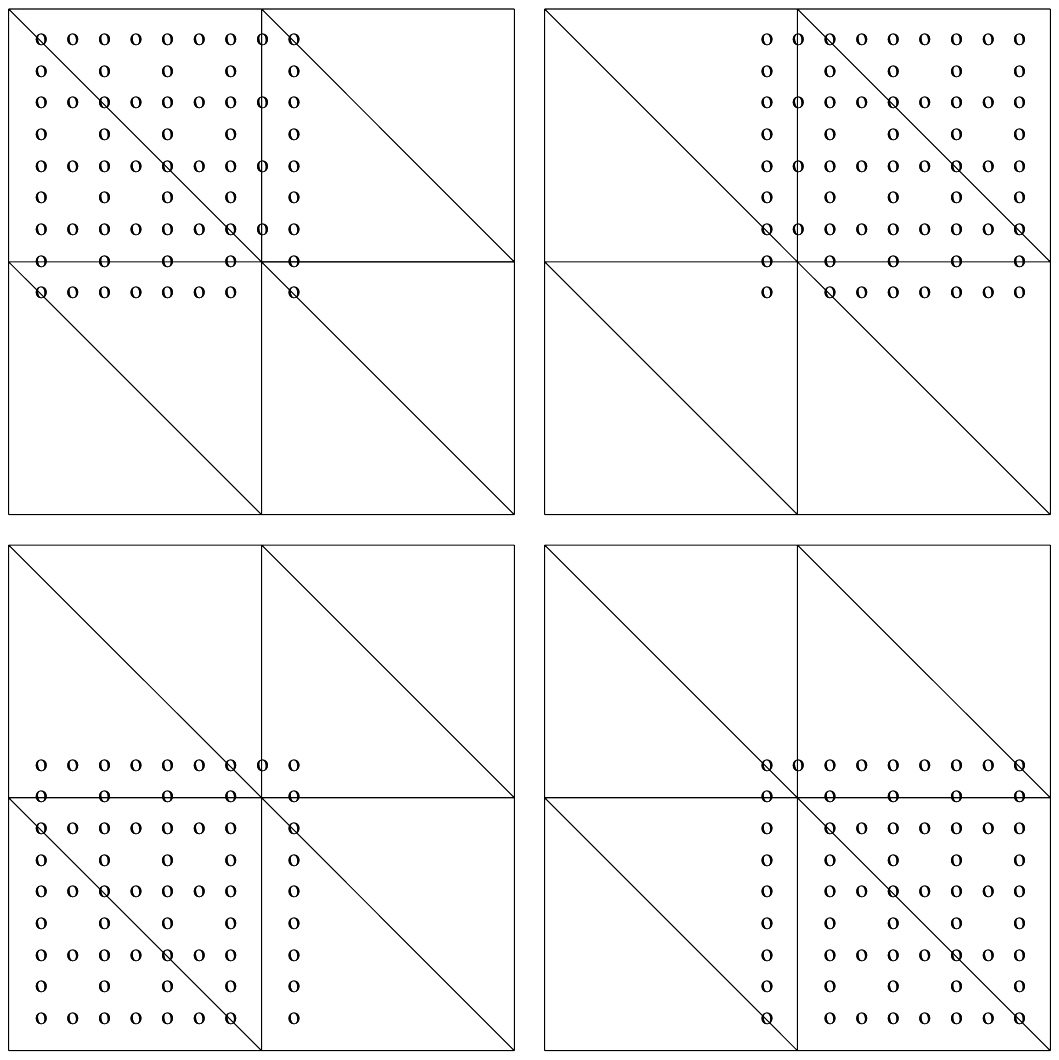}}
\end{picture}
\caption{Left: The level three function nodes in the four subdomains, where $\Omega=(0,1)^2$;
  Right: The 4-subdomain nodes for $\Omega=
   (0,1)^2\setminus\{\frac 12\}\times(0,\frac 12]$. }
\label{d22}
\end{center}
\end{figure}
 
We list the computer found constants of \eqref{b-d} in Table \ref{t22}.
As proved in the theory,  the constants remain bounded on the non-convex domain  $\Omega=
(0,1)^2\setminus\{\frac 12\}\times(0,\frac 12]$, in Table \ref{t22}.

\begin{table}[ht!]
    \centering   \renewcommand{\arraystretch}{1.1}
    \caption{The bounds for 4-subdomain small-overlap DD shown as in Figure \ref{d22}.
     }\label{t22} 
  \begin{tabular}{c|cc|cc}
  \hline
  Grid & $C_{\t{\tiny low}}$ in \eqref{b-d}&$C_{\t{\tiny high}}$ in \eqref{b-d} & $C_{\t{\tiny low}}$ in \eqref{b-d}&$C_{\t{\tiny high}}$ in \eqref{b-d} \\
  \hline
   &\multicolumn{2}{c|}{On $\Omega=(0,1)^2$.}  &\multicolumn{2}{c}{On $\Omega=
   (0,1)^2\setminus\{\frac 12\}\times(0,\frac 12]$.} \\ \hline  
   2 &           3.575004 &           5.000000 &           3.628239 &           5.000000 \\
   3 &           2.609161 &           4.614568 &           2.641507 &           4.611478 \\
   4 &           3.750690 &           4.192807 &           3.381828 &           4.191696 \\
   5 &           4.932503 &           4.049440 &           4.015695 &           4.049120 \\
   6 &           5.554555 &           4.012045 &           4.459404 &           4.011962 \\
  \hline
  \end{tabular}%
  \end{table}%

  \subsection{Tetrahedral N\'{e}d\'{e}lec element}
We solve the equation
\an{\label{e2} \ad{ \bb{curl} \, \bb{curl}\, \b u_h + \b u_h &= \b f \quad &&\t{in } \Omega, \\
                    \b u_h\times \b n &= g   \quad &&\t{on } \partial \Omega, } }
on the two $3D$ domains
\a{ \Omega=(0,2)^3, \qquad \t{ or } \quad \  (0,2)^3\setminus\{1\}\times[1,2)^2. }
The exact solution of \eqref{e2} is chosen as
\an{\label{s3} \b u=\p{ x^2  \\ x^2 \\ y^2 }.  }
In both cases, the meshes used in the computation
    are uniform tetrahedral meshes,  as shown in Figure \ref{g-3d}.
The results are listed in Table \ref{t31}, where we can see that the finite element
  solution converges at the optimal order in both norms on both domains.

\begin{figure}[ht]
\begin{center}
 \setlength\unitlength{1pt}
 \begin{picture}(220,122)(0,3)
    \put(0,0){\begin{picture}(110,110)(0,0)\put(0,102){Grid 1:}
       \multiput(0,0)(40,0){3}{\line(0,1){80}}  \multiput(0,0)(0,40){3}{\line(1,0){80}}
       \multiput(0,80)(40,0){3}{\line(1,1){20}} \multiput(0,80)(10,10){3}{\line(1,0){80}}
       \multiput(80,0)(0,40){3}{\line(1,1){20}}  \multiput(80,0)(10,10){3}{\line(0,1){80}}
 \put(80,0){\line(-1,1){80}}\put(80,0){\line(1,5){20}}\put(80,80){\line(-3,1){60}}
 \multiput(40,0)(40,40){2}{\line(-1,1){40}}  \multiput(80,40)(10,-30){2}{\line(1,5){10}}
 \multiput(40,80)(50,10){2}{\line(-3,1){30}}
      \end{picture}}
    \put(130,0){\begin{picture}(110,110)(0,0)\put(0,102){Grid 2:}
       \multiput(0,0)(20,0){5}{\line(0,1){80}}  \multiput(0,0)(0,20){5}{\line(1,0){80}}
       \multiput(0,80)(20,0){5}{\line(1,1){20}} \multiput(0,80)(5,5){5}{\line(1,0){80}}
       \multiput(80,0)(0,20){5}{\line(1,1){20}}  \multiput(80,0)(5,5){5}{\line(0,1){80}}
 \put(80,0){\line(-1,1){80}}\put(80,0){\line(1,5){20}}\put(80,80){\line(-3,1){60}}
 \multiput(40,0)(40,40){2}{\line(-1,1){40}}  \multiput(80,40)(10,-30){2}{\line(1,5){10}}
 \multiput(40,80)(50,10){2}{\line(-3,1){30}}
 \multiput(20,0)(60,60){2}{\line(-1,1){20}}   \multiput(60,0)(20,20){2}{\line(-1,1){60}}
  \multiput(80,60)(15,-45){2}{\line(1,5){5}} \multiput(80,20)(5,-15){2}{\line(1,5){15}}
  \multiput(20,80)(75,15){2}{\line(-3,1){15}}\multiput(60,80)(25,5){2}{\line(-3,1){45}}
      \end{picture}} 
    \end{picture}
    \end{center}
\caption{ The first two grids for the computation in Tables \ref{t31}--\ref{t32}.  }
\label{g-3d}
\end{figure}
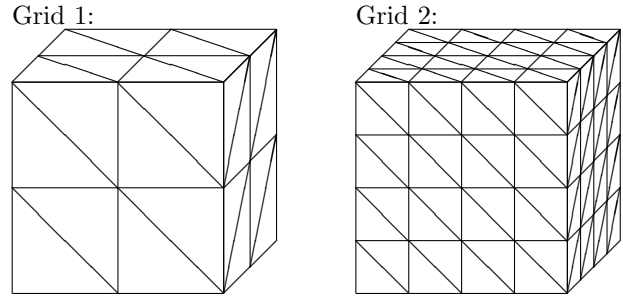

\begin{table}[ht]
    \centering   
    \caption{ Error profile for \eqref{s3}  on grids as shown in Figure \ref{g-3d}.
     }\label{t31} 
  \begin{tabular}{c|cc|cc|cc|cc}
  \hline
  Grid & Error 1 & Order   
        & Error 2 & Order & Error 1 & Order   
        & Error 2 & Order  \\
  \hline
   &\multicolumn{4}{c|}{On $\Omega=(0,2)^3$.} &\multicolumn{4}{c}{On $\Omega=
   (0,2)^3\setminus\{1\}\times[1,2)^2$.} \\ \hline
   1&    8.24E-2& 0.0&    4.57E-1& 0.0 &    8.53E-2& 0.0&    4.22E-1& 0.0 \\
   2&    2.82E-2& 1.5&    3.13E-1& 0.5 &    2.82E-2& 1.6&    3.04E-1& 0.5 \\
   3&    7.75E-3& 1.9&    1.75E-1& 0.8 &    7.75E-3& 1.9&    1.73E-1& 0.8 \\
   4&    2.01E-3& 1.9&    9.21E-2& 0.9 &    2.01E-3& 1.9&    9.16E-2& 0.9 \\
   5&    5.09E-4& 2.0&    4.71E-2& 1.0 &    5.10E-4& 2.0&    4.70E-2& 1.0 \\
   6&    1.28E-4& 2.0&    2.38E-2& 1.0 &    1.29E-4& 2.0&    2.38E-2& 1.0 \\
  \hline
  \end{tabular}%
  \end{table}%

We perform domain decomposition iterations with eight subdomains for both domains of a cube and a cube with
   a cut. The eight subdomains
   are the eight unit cubes in the left graph of Figure \ref{g-3d}.   
We list the computer found constants of \eqref{b-d} in Table \ref{t32}.
As proved in the theory,  the constants remain bounded on the non-convex domain  $\Omega=
    (0,2)^3\setminus\{1\}\times[1,2)^2$, in Table \ref{t32}.
It seems the $C_{\t{\tiny low}}$ in Table \ref{t32} may keep grow.
It would break the theory only when $C_{\t{\tiny low}}$ decreases to $0$.
We note again that we improved previous theoretic lower bound from $O((1+H/\delta)^2)$
   to $O(1+H/\delta)$.
The $C_{\t{\tiny low}}$ in the $2D$ examples seems to confirm that 
   $O(1+H/\delta)$ is the optimal lower bound.
But we are not sure if the computation is done on high enough levels to
   enter the asymptotic range, or if the lower bound $O(1+H/\delta)$ 
    can be further improved in theory for $3D$ tetrahedral edge elements.

\begin{table}[ht!]
    \centering   \renewcommand{\arraystretch}{1.1}
    \caption{The bounds for 8-subdomain DD on meshes shown as in Figure \ref{g-3d}.
     }\label{t32} 
  \begin{tabular}{c|cc|cc}
  \hline
  Grid & $C_{\t{\tiny low}}$ in \eqref{b-d}&$C_{\t{\tiny high}}$ in \eqref{b-d} & $C_{\t{\tiny low}}$ in \eqref{b-d}&$C_{\t{\tiny high}}$ in \eqref{b-d} \\
  \hline
   &\multicolumn{2}{c|}{On $\Omega=(0,2)^3$.}  &\multicolumn{2}{c}{On $\Omega=
   (0,2)^3\setminus\{1\}\times[1,2)^2$.} \\ \hline  
   1 &           1.525539 &           5.846509 &           1.662208 &           5.543800 \\
   2 &           1.792536 &           8.408068 &           1.815730 &           8.381018 \\
   3 &           2.845118 &           8.410624 &           2.856422 &           8.390277 \\
   4 &           4.741863 &           8.410624 &           4.760704 &           8.390277 \\
  \hline
  \end{tabular}%
  \end{table}%


\begin{thebibliography}{10}

    \bibitem{MR1401938}
    Douglas~N. Arnold, Richard~S. Falk, and R.~Winther.
    \newblock Preconditioning in {$H({\rm div})$} and applications.
    \newblock {\em Math. Comp.}, 66(219):957--984, 1997.
    
    \bibitem{MR1754719}
    Douglas~N. Arnold, Richard~S. Falk, and Ragnar Winther.
    \newblock Multigrid in {$H({\rm div})$} and {$H({\rm curl})$}.
    \newblock {\em Numer. Math.}, 85(2):197--217, 2000.
    
    \bibitem{MR2143847}
    Alain Bossavit.
    \newblock Discretization of electromagnetic problems: the ``generalized finite differences'' approach.
    \newblock In {\em Handbook of numerical analysis. {V}ol. {XIII}}, Handb. Numer. Anal., XIII, pages 105--197. North-Holland, Amsterdam, 2005.
    
    \bibitem{MR1771050}
    Susanne~C. Brenner.
    \newblock Lower bounds for two-level additive {S}chwarz preconditioners with small overlap.
    \newblock {\em SIAM J. Sci. Comput.}, 21(5):1657--1669, 2000.
    \newblock Iterative methods for solving systems of algebraic equations (Copper Mountain, CO, 1998).
    
    \bibitem{MR3859167}
    Susanne~C. Brenner and Duk-Soon Oh.
    \newblock Multigrid methods for {$H({\rm div})$} in three dimensions with nonoverlapping domain decomposition smoothers.
    \newblock {\em Numer. Linear Algebra Appl.}, 25(5):e2191, 14, 2018.
    
    \bibitem{MR3989901}
    Susanne~C. Brenner and Duk-Soon Oh.
    \newblock A smoother based on nonoverlapping domain decomposition methods for {$H({\rm div})$} problems: a numerical study.
    \newblock In {\em Domain decomposition methods in science and engineering {XXIV}}, volume 125 of {\em Lect. Notes Comput. Sci. Eng.}, pages 523--531. Springer, Cham, 2018.
    
    \bibitem{BF:2022:MGdeRham}
    Pablo~D. Brubeck and Patrick~E. Farrell.
    \newblock Multigrid solvers for the de rham complex with optimal complexity in polynomial degree, 2022.
    \newblock arXiv:2211.14284.
    
    \bibitem{MR1302685}
    Z.~Cai, R.~Lazarov, T.~A. Manteuffel, and S.~F. McCormick.
    \newblock First-order system least squares for second-order partial differential equations. {I}.
    \newblock {\em SIAM J. Numer. Anal.}, 31(6):1785--1799, 1994.
    
    \bibitem{MR2642330}
    Zhiqiang Cai, Charles Tong, Panayot~S. Vassilevski, and Chunbo Wang.
    \newblock Mixed finite element methods for incompressible flow: stationary {S}tokes equations.
    \newblock {\em Numer. Methods Partial Differential Equations}, 26(4):957--978, 2010.
    
    \bibitem{MR3454359}
    Juan~G. Calvo.
    \newblock A {BDDC} algorithm with deluxe scaling for {$H({\rm curl})$} in two dimensions with irregular subdomains.
    \newblock {\em Math. Comp.}, 85(299):1085--1111, 2016.
    
    \bibitem{MR4064357}
    Juan~G. Calvo.
    \newblock A new coarse space for overlapping {S}chwarz algorithms for {$H(\rm curl)$} problems in three dimensions with irregular subdomains.
    \newblock {\em Numer. Algorithms}, 83(3):885--899, 2020.
    
    \bibitem{MR3242973}
    Clark~R. Dohrmann and Olof~B. Widlund.
    \newblock Some recent tools and a {BDDC} algorithm for 3{D} problems in {$H({\rm curl})$}.
    \newblock In {\em Domain decomposition methods in science and engineering {XX}}, volume~91 of {\em Lect. Notes Comput. Sci. Eng.}, pages 15--25. Springer, Heidelberg, 2013.
    
    \bibitem{MR3465088}
    Clark~R. Dohrmann and Olof~B. Widlund.
    \newblock A {BDDC} algorithm with deluxe scaling for three-dimensional {$H({\bf curl})$} problems.
    \newblock {\em Comm. Pure Appl. Math.}, 69(4):745--770, 2016.
    
    \bibitem{MR1273155}
    Maksymilian Dryja and Olof~B. Widlund.
    \newblock Domain decomposition algorithms with small overlap.
    \newblock {\em SIAM J. Sci. Comput.}, 15(3):604--620, 1994.
    \newblock Iterative methods in numerical linear algebra (Copper Mountain Resort, CO, 1992).
    
    \bibitem{MR3246803}
    Richard~S. Falk and Ragnar Winther.
    \newblock Local bounded cochain projections.
    \newblock {\em Math. Comp.}, 83(290):2631--2656, 2014.
    
    \bibitem{MR1615161}
    R.~Hiptmair.
    \newblock Multigrid method for {$\mathbf H({\rm div})$} in three dimensions.
    \newblock {\em Electron. Trans. Numer. Anal.}, 6:133--152, 1997.
    \newblock Special issue on multilevel methods (Copper Mountain, CO, 1997).
    
    \bibitem{MR1654571}
    R.~Hiptmair.
    \newblock Multigrid method for {M}axwell's equations.
    \newblock {\em SIAM J. Numer. Anal.}, 36(1):204--225, 1999.
    
    \bibitem{HP:2019:RegularDecomp}
    Ralf Hiptmair and Clemens Pechstein.
    \newblock Regular decompositions of vector fields - continuous, discrete, and structure-preserving, 2019.
    
    \bibitem{MR4398318}
    Ralf Hiptmair and Clemens Pechstein.
    \newblock Discrete regular decompositions of tetrahedral discrete 1-forms.
    \newblock In {\em Maxwell's equations---analysis and numerics}, volume~24 of {\em Radon Ser. Comput. Appl. Math.}, pages 199--258. De Gruyter, Berlin, [2019] \copyright 2019.
    
    \bibitem{MR4143285}
    Ralf Hiptmair and Clemens Pechstein.
    \newblock A review of regular decompositions of vector fields: continuous, discrete, and structure-preserving.
    \newblock In {\em Spectral and high order methods for partial differential equations---{ICOSAHOM} 2018}, volume 134 of {\em Lect. Notes Comput. Sci. Eng.}, pages 45--60. Springer, Cham, [2020] \copyright 2020.
    
    \bibitem{MR1838270}
    Ralf Hiptmair and Andrea Toselli.
    \newblock Overlapping and multilevel {S}chwarz methods for vector valued elliptic problems in three dimensions.
    \newblock In {\em Parallel solution of partial differential equations ({M}inneapolis, {MN}, 1997)}, volume 120 of {\em IMA Vol. Math. Appl.}, pages 181--208. Springer, New York, 2000.
    
    \bibitem{MR2361899}
    Ralf Hiptmair and Jinchao Xu.
    \newblock Nodal auxiliary space preconditioning in {${\bf H}({\bf curl})$} and {${\bf H}({\rm div})$} spaces.
    \newblock {\em SIAM J. Numer. Anal.}, 45(6):2483--2509, 2007.
    
    \bibitem{MR2035002}
    Qiya Hu and Jun Zou.
    \newblock A nonoverlapping domain decomposition method for {M}axwell's equations in three dimensions.
    \newblock {\em SIAM J. Numer. Anal.}, 41(5):1682--1708, 2003.
    
    \bibitem{MR2536904}
    Tzanio~V. Kolev and Panayot~S. Vassilevski.
    \newblock Parallel auxiliary space {AMG} for {$H({\rm curl})$} problems.
    \newblock {\em J. Comput. Math.}, 27(5):604--623, 2009.
    
    \bibitem{MR3029843}
    Tzanio~V. Kolev and Panayot~S. Vassilevski.
    \newblock Parallel auxiliary space {AMG} solver for {$H({\rm div})$} problems.
    \newblock {\em SIAM J. Sci. Comput.}, 34(6):A3079--A3098, 2012.
    
    \bibitem{MR2372343}
    Young-Ju Lee, Jinbiao Wu, Jinchao Xu, and Ludmil Zikatanov.
    \newblock Robust subspace correction methods for nearly singular systems.
    \newblock {\em Math. Models Methods Appl. Sci.}, 17(11):1937--1963, 2007.
    
    \bibitem{LXZ:2023:OSSharp}
    Qigang Liang, Xuejun Xu, and Shangyou Zhang.
    \newblock On a sharp estimate of overlapping schwarz methods in $\mathrm{H}(\mathrm{curl})$ and $\mathrm{H}(\mathrm{div})$, 2023.
    \newblock arXiv:2304.10026.
    
    \bibitem{MR1451113}
    Ping Lin.
    \newblock A sequential regularization method for time-dependent incompressible {N}avier-{S}tokes equations.
    \newblock {\em SIAM J. Numer. Anal.}, 34(3):1051--1071, 1997.
    
    \bibitem{MR1135758}
    Peter~B. Monk.
    \newblock A mixed method for approximating {M}axwell's equations.
    \newblock {\em SIAM J. Numer. Anal.}, 28(6):1610--1634, 1991.
    
    \bibitem{MR3243012}
    Duk-Soon Oh.
    \newblock An alternative coarse space method for overlapping {S}chwarz preconditioners for {R}aviart-{T}homas vector fields.
    \newblock In {\em Domain decomposition methods in science and engineering {XX}}, volume~91 of {\em Lect. Notes Comput. Sci. Eng.}, pages 361--367. Springer, Heidelberg, 2013.
    
    \bibitem{MR3033012}
    Duk-Soon Oh.
    \newblock An overlapping {S}chwarz algorithm for {R}aviart-{T}homas vector fields with discontinuous coefficients.
    \newblock {\em SIAM J. Numer. Anal.}, 51(1):297--321, 2013.
    
    \bibitem{MR3618847}
    Duk-Soon Oh.
    \newblock A {BDDC} preconditioner for problems posed in {$H({\rm div})$} with deluxe scaling.
    \newblock In {\em Domain decomposition methods in science and engineering {XXII}}, volume 104 of {\em Lect. Notes Comput. Sci. Eng.}, pages 355--361. Springer, Cham, 2016.
    
    \bibitem{Oh:2022:MGHCurl}
    Duk-Soon Oh.
    \newblock Multigrid methods for 3{$D$} ${H}(\mathbf{curl})$ problems with nonoverlapping domain decomposition smoothers, 2022.
    \newblock submitted, arXiv:2205.05840.
    
    \bibitem{Oh:MGHcurlNE}
    Duk-Soon Oh.
    \newblock Smoothers based on nonoverlapping domain decomposition methods for {$H({\rm curl})$} problems: A numerical study.
    \newblock {\em J. Korean Soc. Ind. Appl. Math.}, 26(4):323--332, 2022.
    
    \bibitem{MR3739213}
    Duk-Soon Oh, Olof~B. Widlund, Stefano Zampini, and Clark~R. Dohrmann.
    \newblock B{DDC} algorithms with deluxe scaling and adaptive selection of primal constraints for {R}aviart-{T}homas vector fields.
    \newblock {\em Math. Comp.}, 87(310):659--692, 2018.
    
    \bibitem{MR1011446}
    L.~Ridgway Scott and Shangyou Zhang.
    \newblock Finite element interpolation of nonsmooth functions satisfying boundary conditions.
    \newblock {\em Math. Comp.}, 54(190):483--493, 1990.
    
    \bibitem{MR1794350}
    Andrea Toselli.
    \newblock Overlapping {S}chwarz methods for {M}axwell's equations in three dimensions.
    \newblock {\em Numer. Math.}, 86(4):733--752, 2000.
    
    \bibitem{MR2103209}
    Andrea Toselli.
    \newblock Domain decomposition methods of dual-primal {FETI} type for edge element approximations in three dimensions.
    \newblock {\em C. R. Math. Acad. Sci. Paris}, 339(9):673--678, 2004.
    
    \bibitem{MR2193972}
    Andrea Toselli.
    \newblock Dual-primal {FETI} algorithms for edge finite-element approximations in 3{D}.
    \newblock {\em IMA J. Numer. Anal.}, 26(1):96--130, 2006.
    
    \bibitem{MR2169505}
    Andrea Toselli and Xavier Vasseur.
    \newblock Robust and efficient {FETI} domain decomposition algorithms for edge element approximations.
    \newblock {\em COMPEL}, 24(2):396--407, 2005.
    
    \bibitem{MR2104179}
    Andrea Toselli and Olof Widlund.
    \newblock {\em Domain decomposition methods---algorithms and theory}, volume~34 of {\em Springer Series in Computational Mathematics}.
    \newblock Springer-Verlag, Berlin, 2005.
    
    \bibitem{MR1759911}
    Barbara~I. Wohlmuth, Andrea Toselli, and Olof~B. Widlund.
    \newblock An iterative substructuring method for {R}aviart-{T}homas vector fields in three dimensions.
    \newblock {\em SIAM J. Numer. Anal.}, 37(5):1657--1676, 2000.
    
    \bibitem{MR1193013}
    Jinchao Xu.
    \newblock Iterative methods by space decomposition and subspace correction.
    \newblock {\em SIAM Rev.}, 34(4):581--613, 1992.
    
    \bibitem{MR3718365}
    Stefano Zampini.
    \newblock Adaptive {BDDC} deluxe methods for {$\rm H(curl)$}.
    \newblock In {\em Domain decomposition methods in science and engineering {XXIII}}, volume 116 of {\em Lect. Notes Comput. Sci. Eng.}, pages 285--292. Springer, Cham, 2017.
    
    \bibitem{MR3989859}
    Stefano Zampini, Panayot Vassilevski, Veselin Dobrev, and Tzanio Kolev.
    \newblock Balancing domain decomposition by constraints algorithms for curl-conforming spaces of arbitrary order.
    \newblock In {\em Domain decomposition methods in science and engineering {XXIV}}, volume 125 of {\em Lect. Notes Comput. Sci. Eng.}, pages 103--116. Springer, Cham, 2018.
    
    \end{thebibliography}
\end{document}